%% file: Markov_G_fan.tex
\documentclass[11pt,a4paper,reqno]{amsart}
\fontsize{6.5pt}{6.5pt}\selectfont 
\makeatletter
\def\@evenhead{{\fontsize{6.5pt}{6.5pt}\selectfont \hfil \leftmark\hfil\thepage}}
\def\@oddhead{{\fontsize{6.5pt}{6.5pt}\selectfont\hfil\rightmark \hfil\thepage}}
\makeatother
\usepackage{amsfonts}
\usepackage{booktabs}
\usepackage{amsmath,amssymb,amsthm,amsxtra}
\usepackage{float}
\usepackage{amssymb}
\usepackage{mathabx}
\usepackage{bbm}
\usepackage[colorlinks, linkcolor=blue,anchorcolor=Periwinkle,
citecolor=Red,urlcolor=Emrald]{hyperref}
\usepackage[usenames,dvipsnames]{xcolor}
\usepackage{enumitem}
\usepackage{geometry,array} 
\usepackage{graphicx}
\usepackage{subfigure}
\usepackage{bookmark}
\usepackage{tikz}\usetikzlibrary{matrix,intersections, calc, arrows.meta}
\usepackage{url}
\usepackage{dsfont}
\usepackage[colorinlistoftodos]{todonotes}
\usepackage{tablists} \restorelistitem
\usepackage{bm}

\usetikzlibrary{decorations.markings}
\tikzset{->-/.style={decoration={  markings,  mark=at position #1 with
    {\arrow{>}}},postaction={decorate}}}
\tikzset{-<-/.style={decoration={  markings,  mark=at position #1 with
    {\arrow{<}}},postaction={decorate}}}
\usepackage{extarrows}
\usepackage[all]{xy}
\usepackage{setspace}
\usepackage{thmtools}
\usepackage{thm-restate}
\usepackage{hyperref}
\usepackage{cleveref}
\usepackage{multirow}
\usepackage{ifthen}
\usepackage{tikz-3dplot}
\usepackage{comment}
\usepackage{mathrsfs}
\usepackage{subfiles}
\usepackage{aligned-overset}

\usepackage{pgfplots}
\pgfplotsset{compat=1.18}

\usepackage{caption}

\newcommand{\mbR}{\mathbb{R}}

\theoremstyle{plain}
\newtheorem{theorem}{Theorem}[section]

\newtheorem{lemma}[theorem]{Lemma}
\newtheorem{corollary}[theorem]{Corollary}
\newtheorem{proposition}[theorem]{Proposition}

\theoremstyle{definition}
\newtheorem{definition}[theorem]{Definition}

\newtheorem{example}[theorem]{Example}

\newtheorem{remark}[theorem]{Remark}

\newtheorem{question}[theorem]{Question}

\numberwithin{equation}{section}
\newtheorem{definition-proposition}[theorem]{Definition-Proposition}

\newcommand{\Clinearmap}{\Phi}
\newcommand{\Glinearmap}{\Psi}

\begin{document}
\title{Fractal phenomenon in $c$- and $g$-vectors of the Markov quiver}

\author{Ryota Akagi}
\address{Graduate School of Mathematics\\ nagoya University\\Chikusa-ku\\ Nagoya\\464-8602\\ Japan.}
\email{ryota.akagi.e6@math.nagoya-u.ac.jp}

\author{Zhichao Chen}
\address{School of Mathematical Sciences\\ University of Science and Technology of China \\ Hefei, Anhui 230026, P. R. China.}
\email{czc98@mail.ustc.edu.cn}
\maketitle
\begin{abstract}
 We study the $C$- and $G$-patterns associated with rank $3$ skew-symmetrizable matrices of $B$-invariant type, including the Markov quiver. Motivated by the self-contained simple mutations in Markov-type cluster algebras, we prove that large classes of subpatterns of modified $c$- and $g$-vectors are linearly isomorphic, yielding a fractal structure of the corresponding $G$-fan. We further derive explicit recursive formulas for all modified $c$- and $g$-vectors in terms of integer pairs satisfying a recursion analogous to the Calkin-Wilf tree, which leads to a parameterization by coprime integers. As an application, we describe all connected components of the complement of the support of the $G$-fan, and show that they are generated recursively by three kinds of linear maps.
\end{abstract}

\tableofcontents
\section{Introduction}

\subsection{Background}
The $c$- and $g$-vectors are interesting combinatorial objects which arise from the cluster algebra theory in \cite{FZ02}. They were originally introduced by \cite{FZ07} as degree vectors of cluster variables and their coefficients. After that, many interesting phenomena and applications were found, and now they are regarded as one of the most fundamental tools in cluster algebra theory.
\par
The most important objects in $c$- and $g$-vectors are the \emph{$G$-fans}, which are simplicial fans consisting of $g$-vectors. Many investigations show that the $G$-fan is a cluster analogue of the Coxeter chamber structure (e.g. \cite{RS16,RS18}). 
\par
In general, particularly in the case of infinite type, the structure of the $G$-fans seems to be rather complicated, and there is only little progress. Observing many examples of rank $3$ in \cite[\S~6.7]{Nak23}, one might suspect that the $G$-fan possesses a ``fractal" structure. (Here, the term ``fractal" refers to the existence of numerous internal isomorphisms, rather than the non-integer dimension associated with the Hausdorff dimension.) The purpose of this paper is to justify this phenomenon in the following simplest infinite case called the {\em $B$-invariant type} of rank $3$.
\subsection{$B$-invariant type}
In this paper, we focus on the following real exchange matrices:
\begin{equation}\label{eq: invariant type}
B=\pm\left(\begin{matrix}
0 & -p' & r\\
p & 0 & -q'\\
-r' & q & 0
\end{matrix}\right),
\end{equation}
where $p,p',q,q',r,r' \in \mathbb{R}_{>0}$ satisfy $pp'=qq'=rr'=4$ and $pqr=p'q'r'$. (As a consequence, we have $pqr=p'q'r'=8$.) 
This is certainly a skew-symmetrizable matrix. For example, if we set $D=\mathrm{diag}(pr',p'r',pr)$, then $DB$ is skew-symmetric.
We say that such a skew-symmetrizable matrix is of {\em $B$-invariant type} of rank $3$.
\par
In the usual cluster theory, we assume that all the entries in exchange matrices are integers. However, the combinatorial aspects of cluster algebras can be generalized to the real entries through $c$- and $g$-vectors \cite{AC25a}. Hence, we consider this more general setting. 
\par
The most important property is that the mutation of this exchange matrix is given by 
\begin{equation}
\mu_i(B)=-B
\end{equation}
for each $i=1,2,3$. By \cite[Thm.~1.1]{CL25}, such situation occurs for the rank $3$ irreducible skew-symmetrizable matrices only when $B$ is given by \eqref{eq: invariant type} and, if the rank is greater than $3$, this never happens. See also \cite[Lem.~9.8, Rem.~9.10]{Aka24}.
\par
We can give the following examples with integer entries.
\begin{equation}\label{eq: integer examples of B invariant type}
B=
\pm\left(\begin{matrix}
0 & -2 & 2\\
2 & 0 & -2\\
-2 & 2 & 0
\end{matrix}\right),
\quad
\pm \left(\begin{matrix}
0 & -4 & 4\\
1 & 0 & -2\\
-1 & 2 & 0
\end{matrix}\right).
\end{equation}
According to \cite[Thm.~1.1]{CL25}, all \emph{integer} skew-symmetrizable matrices of $B$-invariant type are enumerated in \eqref{eq: integer examples of B invariant type} up to simultaneous row and column permutations.
The first one in \eqref{eq: integer examples of B invariant type} is known as the skew-symmetric matrix corresponding to the {\em Markov quiver} in \Cref{fig: Markov quiver}.
\begin{figure}[htbp]
\centering
\begin{minipage}[b]{0.45\textwidth}
\centering
\begin{tikzpicture}[dot/.style={circle, fill, inner sep=1.5pt, outer sep=3pt}, >={Classical TikZ Rightarrow[length=4, width=3]}]
  \node[dot] (1) at ({cos(90)},{sin(60)}) {};
  \node[dot] (2) at ({cos(210)},{sin(210)}) {};
  \node[dot] (3) at ({cos(330)},{sin(330)}) {};
  \draw[->] (1.190)->(2.80);
  \draw[->] (1.240)->(2.30);
  \draw[->] (2.15)->(3.165);
  \draw[->] (2.-35)->(3.215);
  \draw[->] (3.150)->(1.-60);
  \draw[->] (3.100)->(1.-15);
\end{tikzpicture}
\caption{Markov quiver.}\label{fig: Markov quiver}
\end{minipage}
\begin{minipage}[b]{0.54\textwidth}
\centering
\begin{tikzpicture}[scale=0.6]
    \draw (-2,0) to [out=-90, in=-90] coordinate [pos=0.47] (FB4) coordinate [pos=0.15] (FB3) (2,0);    
    \draw (-2,0) to [out=90, in=90] (2,0);  
    \draw (-1,0.2) to [out=-40, in=-140] coordinate [pos=0.2] (SBL) coordinate [pos=0.4] (SB4) coordinate [pos=0.5] (SBM) coordinate [pos=0.8] (SBR) (1,0.2);
    \draw (SBL) to [out=40, in=140] coordinate [pos=0.5] (SAM) coordinate [pos=0.7] (SA7) (SBR);  
    \coordinate (Pun) at ($(0,-1)!0.5!(SBM)$);
    \coordinate (APun) at ($(0,1)!0.5!(SAM)+(0,0.2)$);
    \coordinate (RPun) at (1.5,0);
    \coordinate (LPun) at (-1.5,0);
    \fill (Pun) circle [radius=0.08];
    \draw (FB4) to [out=60, in=-90] (Pun) to [out=90, in =-40] (SB4);
    \draw[dashed] (SB4) to [out=-150, in=150] (FB4);
    \draw (Pun) to [out=0, in=-90] (RPun) to [out=90, in=0] (APun) to [out=180, in=90] (LPun) to [out=-90, in=0] (Pun);
    \draw (Pun) to [out=30, in=-30] (1.1,0.4) to [out=150, in=20] (SAM);
    \draw[dashed] (SAM) to [out=160, in=30] (-1.5,0.4) to [out=-150, in=110] (FB3);
    \draw (FB3) to [out=-30, in=-150] (Pun);
\end{tikzpicture}
\caption{Once punctured torus.}
\label{fig: once punctured torus}
\end{minipage}
\end{figure}

\subsection{Progress in the Markov-type cluster algebra}
In cluster algebra theory, the Markov quiver appears as an example of many different classes. For example, we can find the Markov quiver and, equivalently, the once punctured torus in the following papers.
\begin{itemize}
    \item In \cite{FST08}, the geometric realization of the corresponding cluster algebra is given by triangulations of the once punctured torus (\Cref{fig: once punctured torus}).
    \item In \cite{FG16}, the tropical boundary, which is an analogue of the Thurston boundary in the Teichm\"uller theory, associated with the once punctured torus was calculated.
    \item In \cite{DWZ08}, the Jacobi algebra and a nondegenerate quiver with potential associated with the Markov quiver appear. In \cite{Ric15}, the detailed properties of this Jacobi algebra were studied.
    \item In \cite{Pro20}, the mutation of cluster variables was found in the {\em Markov tree}, which is a binary tree that enumerates the Markov triplets. Recently, from this correspondence, some combinatorial tools to study the Markov triplets were imported into the cluster algebra by \cite{BG25}.
\end{itemize}
As previously mentioned, our main motivation is to establish the fractal structure of the $G$-fan. Since this class arises in such a wide range of fields, it is also interesting to study how and why our fractal theorem appears in these applications. This broad applicability further justifies our focus on this particular class.
\par
For the purpose of studying $c$- and $g$-vectors, the relation between $c$- and $g$-vectors and the Farey triplets is obtained by \cite{Cha12}, and an alternative construction was given by \cite{Rea15} from the surface. Although our forthcoming approach is different from the ones in \cite{Cha12, Rea15}, one important result can also be obtained by using the technique in this paper, see \Cref{cor: coprime number expression}.
\subsection{Modified $c$- and $g$-vectors}
Throughout this paper, we always fix one skew-symmetrizable matrix $B \in \mathrm{M}_3(\mathbb{R})$ of $B$-invariant type and its skew-symmetrizer $D=\mathrm{diag}(d_1,d_2,d_3)$.
A sequence $\mathbf{w}=[k_1,\dots,k_r]$ of $1,2,3$ is said to be {\em reduced} if $k_i \neq k_{i+1}$ for any $i=1,\dots,r-1$. By convention, the empty sequence $\emptyset=[\ ]$ is also reduced. We write the set of all reduced sequences by $\mathcal{T}$.
For any reduced sequence $\mathbf{w} \in \mathcal{T}$, we define the $3 \times 3$ matrix $B=(b_{ij}^{\mathbf{w}}) \in \mathrm{M}_3(\mathbb{R})$ as
\begin{equation}
B^{\mathbf{w}}=(-1)^{|\mathbf{w}|}B,
\end{equation}
which corresponds to the mutated exchange matrix $B^{\mathbf{w}}=\mu_{k_r}\cdots\mu_{k_1}(B)$ in general cluster algebra theory.
\par
For any $\mathbf{w}=[k_1,\dots,k_r] \in \mathcal{T}$ and $k=1,2,3$ with $k_r \neq k$, define $\mathbf{w}[k]=[k_1,\dots,k_r,k] \in \mathcal{T}$.
Then, the $c$-vectors $\mathbf{c}_{j}^{\mathbf{w}}=(c_{ij}^{\mathbf{w}})_{i=1}^3 \in \mathbb{R}^3$ and the $g$-vectors $\mathbf{g}_j^{\mathbf{w}}=(g_{ij}^{\mathbf{w}})_{i=1}^3 \in \mathbb{R}^3$ is defined by the initial condition $\mathbf{c}_j^{\emptyset}=\mathbf{g}_j^{\emptyset}=\mathbf{e}_j$ (the $j$th unit vector) and the following recursion:
\begin{equation}\label{eq: recursion of C}
\begin{aligned} 
c_{ij}^{\mathbf{w}[k]}&=\begin{cases}
-c_{ik}^{\mathbf{w}} & j=k,\\
c_{ij}^{\mathbf{w}} + c_{ik}^{\mathbf{w}}\max(b_{kj}^{\mathbf{w}},0)+\max(-c_{ik}^{\mathbf{w}},0)b_{kj}^{\mathbf{w}} & j \neq k,
\end{cases},
\\
g_{ij}^{\mathbf{w}[k]}&=\begin{cases}
-g_{ik}^{\mathbf{w}}+\sum_{l=1}^{3}g_{il}^{\mathbf{w}}\max(-b_{lk}^{\mathbf{w}},0)-\sum_{l=1}^3\max(c_{lk}^{\mathbf{w}},0)b_{lk}^{\emptyset} & j=k,
\\
g_{ij}^{\mathbf{w}} & j \neq k.
\end{cases}
\end{aligned}
\end{equation}
We will not treat this formula directly. Later, we introduce another expression of this recursion formulated in terms of the $c$- and $g$-vectors in \Cref{lem: mutation of modified vectors}.
\par
Instead of the usual $c$- and $g$-vectors, we employ the {\em modified $c$-vectors} $\tilde{\mathbf{c}}_{j}^{\mathbf{w}}$ and the {\em modified $g$-vectors} $\tilde{\mathbf{g}}_{j}^{\mathbf{w}}$, which are defined by 
\begin{equation}
\tilde{\mathbf{c}}_{j}^{\mathbf{w}}=\frac{1}{\sqrt{d_j}}\mathbf{c}_j^{\mathbf{w}},
\quad
\tilde{\mathbf{g}}_j^{\mathbf{w}}=\frac{1}{\sqrt{d_j}}\mathbf{g}_j^{\mathbf{w}}.
\end{equation}
Hence, the essential structure does not change. On the other hand, by considering this modification, all cases in this paper follow the same recursions as in \Cref{lem: mutation of modified vectors}.
In accordance with the above modification, we introduce the {\em modified standard vector} $\tilde{\mathbf{e}}_j$ ($j=1,2,3$) by
\begin{equation}
\tilde{\mathbf{e}}_j=\frac{1}{\sqrt{d}_j}\mathbf{e}_j.
\end{equation}
\subsection{Additional notations}
To make the fractal structure clearer, we use the maps $K,S,T: \mathcal{T}\setminus\{\emptyset\} \to \{1,2,3\}$ to express the indices $1,2,3$, see \Cref{sec: tropical signs}. This notation was introduced in \cite{AC25b} for the recursion of tropical signs, see also \Cref{thm: recursion for tropical signs}. Let $\mathcal{M}$ be the set of all words consisting of two letters $S$ and $T$. As in the usual sense, this set $\mathcal{M}$ has the standard monoid structure. Then, except for the empty sequence $\emptyset$, each reduced sequence $\mathbf{w} \in \mathcal{T} \setminus \{\emptyset\}$ can be expressed as $\mathbf{w}=[i]X$ uniquely, where $i=1,2,3$ is the initial mutation direction of $\mathbf{w}$ and $X \in \mathcal{M}$. By using these notations, each $c$-vector and $g$-vector can be expressed as $\mathbf{c}_{M}^{[i]X}=\mathbf{c}_{M([i]X)}^{[i]X}$ and $\mathbf{g}_{M}^{[i]X}=\mathbf{g}_{M([i]X)}^{[i]X}$, where $i=1,2,3$, $X \in \mathcal{M}$, and $M=K,S,T$.
\par
We introduce the partial order $\leq$ on $\mathcal{T}$ by
\begin{equation}
[k_1,\dots,k_r] \leq [l_1,\dots,l_{r'}] \Longleftrightarrow r \leq r'\ \textup{and}\ k_{i}=l_{i}\ \textup{for any}\ i=1,\dots,r.
\end{equation}
For any $\mathbf{w} \in \mathcal{T} \setminus \{\emptyset\}$, define $\mathcal{T}^{\geq \mathbf{w}}=\{\mathbf{u} \mid \mathbf{w} \leq \mathbf{u}\}$. Note that $\mathcal{T}^{\geq [i]}=\{[i]X \mid X \in \mathcal{M}\}$ holds. We decompose $\mathcal{T}^{\geq [i]}$ into the following two kinds of sets:
\begin{itemize}
    \item {\em trunk}: $\mathcal{T}^{<[i]S^{\infty}}=\{[i]S^n \mid n \in \mathbb{Z}_{\geq 0}\}$.
    \item {\em branch}: $\mathcal{T}^{\geq [i]X}$ where $X \in \mathcal{M}$ contains at least one $T$.
\end{itemize}
We can also see \Cref{fig: example of tropical signs} as an example. Depending on the initial exchange matrix $B$ and the initial mutation direction $i=1,2,3$, we set $k_0=K([i])$, $s_0=S([i])$, and $t_0=T([i])$, see \Cref{tab: List of initial indices}.
\subsection{Main theorems}
Our first theorem presents the fractal structure appearing in $c$- and $g$-vectors. We say that two sequences $\mathbf{w},\mathbf{u} \in \mathcal{T}\setminus\{\emptyset\}$ are {\em admissible} if both belong to trunks or branches. Then, we can consider the sub modified $C$-patterns $\widetilde{\mathbf{C}}^{\geq \mathbf{w}}(B)=\{(\tilde{\mathbf{c}}_{1}^{\mathbf{v}},\tilde{\mathbf{c}}_2^{\mathbf{v}},\tilde{\mathbf{c}}_3^{\mathbf{v}})\}_{\mathbf{v} \geq \mathbf{w}}$ and $\widetilde{\mathbf{C}}^{\geq \mathbf{u}}(B)$, and the sub modified $G$-pattern $\widetilde{\mathbf{G}}^{\geq \mathbf{w}}(B)=\{(\tilde{\mathbf{g}}_{1}^{\mathbf{v}},\tilde{\mathbf{g}}_2^{\mathbf{v}},\tilde{\mathbf{g}}_3^{\mathbf{v}})\}_{\mathbf{v} \geq \mathbf{w}}$ and $\widetilde{\mathbf{G}}^{\geq \mathbf{u}}(B)$. Then, our first theorem serves as the isomorphism between these two patterns.
\begin{theorem}[\Cref{thm: inner isomorphisms between subfans in the Markov case}]\label{thm: intro inner isomorphisms between subfans in the Markov case}
For any admissible pair $\mathbf{w},\mathbf{u} \in \mathcal{T}\setminus\{\emptyset\}$, there exist linear isomorphisms $\Clinearmap_{\mathbf{w}}^{\mathbf{u}}, \Glinearmap_{\mathbf{w}}^{\mathbf{u}} \in \mathrm{GL}(\mathbb{R}^3)$ such that, for any $M=K,S,T$ and $X \in \mathcal{M}$,
\begin{equation}
\Clinearmap_{\mathbf{w}}^{\mathbf{u}}(\tilde{\mathbf{c}}_{M}^{\mathbf{w}X})=\tilde{\mathbf{c}}_{M}^{\mathbf{u}X},
\quad
\Glinearmap_{\mathbf{w}}^{\mathbf{u}}(\tilde{\mathbf{g}}_{M}^{\mathbf{w}X})=\tilde{\mathbf{g}}_{M}^{\mathbf{u}X}.
\end{equation}
\end{theorem}
Thanks to this theorem, we can give many isomorphisms between two subpatterns via certain linear isomorphisms $\Clinearmap_{\mathbf{w}}^{\mathbf{u}}, \Glinearmap_{\mathbf{w}}^{\mathbf{u}} \in \mathrm{GL}(\mathbb{R}^3)$. In particular, suppose that $\mathbf{w}$ is in a branch and $\mathbf{u}=\mathbf{w}X$. Then, $\widetilde{\mathbf{C}}^{\geq \mathbf{w}X}(B)$ and $\widetilde{\mathbf{G}}^{\geq \mathbf{w}X}(B)$ are subpatterns of $\widetilde{\mathbf{C}}^{\geq \mathbf{w}}(B)$ and $\widetilde{\mathbf{G}}^{\geq \mathbf{w}}(B)$, respectively. On the other hand, due to this theorem, these two patterns are isomorphic via $\Clinearmap_{\mathbf{w}}^{\mathbf{w}X}:\widetilde{\mathbf{C}}^{\geq \mathbf{w}}(B) \cong \widetilde{\mathbf{C}}^{\geq \mathbf{w}X}(B)$ and $\Glinearmap_{\mathbf{w}}^{\mathbf{w}X}:\widetilde{\mathbf{G}}^{\geq \mathbf{w}}(B) \cong \widetilde{\mathbf{G}}^{\geq \mathbf{w}X}(B)$. Namely, $\Clinearmap_{\mathbf{w}}^{\mathbf{w}X}$ and $\Glinearmap_{\mathbf{w}}^{\mathbf{w}X}$ can be seen as maps explaining their fractal structure.  We write these maps as $\Clinearmap_{\mathbf{w}}^{X}=\Clinearmap_{\mathbf{w}}^{\mathbf{w}X}$ and $\Glinearmap_{\mathbf{w}}^{X}=\Glinearmap_{\mathbf{w}}^{\mathbf{w}X}$. 
\par
For applications, we take appropriate bases for $c$- and $g$-vectors in \eqref{eq: change of basis for c} and \eqref{eq: change of basis for g}. Originally, these bases are taken as the Jordan bases of $\Clinearmap_{\mathbf{w}}^S$ and $\Glinearmap_{\mathbf{w}}^{S}$. However, these bases are also compatible for $\Clinearmap_{\mathbf{w}}^T$ and $\Glinearmap_{\mathbf{w}}^T$, see \Cref{lem: mutations for eigen vectors}. One important fact is that, if $\mathbf{w}$ is in a branch with the initial mutation direction $i=1,2,3$, the vector $\mathfrak{c}$ (resp. $\mathfrak{g}_i$) defined by \eqref{eq: fixed points of mutations} is invariant under the linear maps $\Clinearmap_{\mathbf{w}}^{X}$ (resp. $\Glinearmap_{\mathrm{w}}^{X}$) for all $X \in \mathcal{M}$. As one of the important corollaries, we obtain the following recursive formulas to calculate all $c$-vectors and $g$-vectors. For the sake of brevity, we only present the formula for $g$-vectors here.
\begin{theorem}[\Cref{thm: recursive formula in a branch}]
Fix a branch $\mathcal{T}^{\geq \mathbf{w}}$. Set $\tilde{\mathbf{v}}_{SK}^{\mathbf{w}}=\tilde{\mathbf{g}}_{K}^{\mathbf{w}}-\tilde{\mathbf{g}}_{S}^{\mathbf{w}}$ and $\tilde{\mathbf{v}}_{TK}^{\mathbf{w}}=\tilde{\mathbf{g}}_{K}^{\mathbf{w}}-\tilde{\mathbf{g}}_{T}^{\mathbf{w}}$. Then, for any $X \in \mathcal{M}$ and $M=K,S,T$, we may express
\begin{equation}
\tilde{\bf g}_{M}^{{\bf w}X}=\mathfrak{g}_i+q_{M;S}^{X}\tilde{\bf v}_{SK}^{\bf w}+q_{M;T}^{X}\tilde{\bf v}_{TK}^{\bf w}
\end{equation}
for some $(q_{M;S}^{X}, q_{M;T}^{X}) \in \mathbb{Z}_{\geq 0}^2$ independent of ${\bf w}$.
Moreover, for any $M=K,S,T$ and $X \in \mathcal{M}$, the coefficients $(a^{X},b^X)=(q_{M;S}^{X},q_{M;T}^{X})$ obey the following recursion independent of $M=K,S,T$:
\begin{equation}\label{eq: intro generating rule}
(a^{M'X},b^{M'X})=\begin{cases}
(a^{X}+b^{X},b^X) & M'=S,\\
(b^X,a^X+b^X) & M'=T.
\end{cases}
\end{equation}
\end{theorem}
We can see \Cref{fig: axbx after 11} as an example of these coefficients $(q_{K;S}^{X},q_{K;T}^{X})$. Here, we would like to emphasize that this generating rule in \eqref{eq: intro generating rule} is so similar (but not the same) as the one of the {\em Calkin-Wilf tree} \cite{CW00}. On the other hand, in \cite{Cha12}, all the $c$- and $g$-vectors are constructed based on the tree structure of the {\em Farey triplets}, which is essentially equivalent to the {\em Stern-Brocot tree}. Although our theorem and \cite{Cha12} focus on the same objects, the construction method is completely different. In \cite{Gyo23}, a similar phenomenon is found through the {\em $d$-vectors}.
\par
As an application for $g$-vectors, note that the family $\{(q_{K;S}^{X},q_{K;T}^{X})\}_{X \in \mathcal{M}}$ satisfies many properties appearing in the Calkin-Wilf tree. In particular, all the pairs of positive coprime numbers appear in this family. Thus, we give a method to enumerate all the $g$-vectors from the properties of numbers, which has already appeared in \cite{Cha12,Rea15}.
\begin{corollary}[{\Cref{cor: coprime number expression}, \cite{Cha12, Rea15}}] \label{cor: intro cor: coprime number expression}
Fix one initial mutation direction $i=1,2,3$. Let $\tilde{\mathbf{v}}_{lm}=\tilde{\mathbf{e}}_{m}-\tilde{\mathbf{e}}_{l}$ for $l,m=1,2,3$.
Then, a vector $\tilde{\mathbf{g}} \in \mathbb{R}^3$ appears in the modified $G$-pattern $\widetilde{\mathbf{G}}^{\geq [i]}(B)$ if and only if it is of the form
\begin{equation}\label{eq: coprime number expression of g vectors}
\tilde{\mathbf{g}}=\mathfrak{g}_i+a\tilde{\mathbf{v}}_{t_0k_0}+b\tilde{\mathbf{v}}_{k_0s_0},
\end{equation}
where $(a,b) \in \mathbb{Z}^2$ satisfies either of the following two conditions:
\begin{itemize}
    \item $a,b \geq 1$ and $\gcd(a,b)=1$.
    \item $(a,b)=(1,0)$ (then $\tilde{\mathbf{g}}=\tilde{\mathbf{e}}_{s_0}$) or $(a,b)=(0,-1)$ (then $\tilde{\mathbf{g}}=\tilde{\mathbf{e}}_{t_0}$).
\end{itemize}
Moreover, all the above vectors are modified $g$-vectors.
\end{corollary}
We can obtain a similar formula for $c$-vectors in \Cref{cor: coprime number expression for c vectors}.
\par
Lastly, we give an application for the $G$-fan $\Delta(B)$, see \Cref{def: G-fan}. Some papers (e.g. \cite{Cha12, Rea15, FG16}) visualized the $G$-fan corresponding to the Markov quiver as in \Cref{fig: Markov G-fan}, but it is still difficult to describe the total structure. Thanks to \Cref{cor: intro cor: coprime number expression}, we can obtain the fact that the $G$-fan is contained in the half space
\begin{equation}
V=\{x_1\tilde{\mathbf{e}}_1+x_2\tilde{\mathbf{e}}_2+x_3\tilde{\mathbf{e}}_3 \mid x_1+x_2+x_3 > 0\} \cup \{\mathbf{0}\}.
\end{equation}
One natural and important question is to describe its complement $V \setminus |\Delta(B)|$. Let $\mathrm{Com}(\Delta(B))$ be the set of all connected components in $V \setminus |\Delta(B)|$. We decompose this set into the three pairwise disjoint subsets $\mathrm{Com}^i(\Delta(B))$ associated with the initial mutation direction $i=1,2,3$. Motivated by \Cref{cor: intro cor: coprime number expression}, we define
\begin{equation}
\mathfrak{G}_i=\{\mathfrak{g}_i+a\tilde{\mathbf{v}}_{t_0k_0}+b\tilde{\mathbf{v}}_{k_0s_0} \mid (a,b) \in \mathbb{Z}_{\geq 1}^2, \gcd(a,b)=1\} \cup \{\tilde{\mathbf{e}}_{s_0}\}.
\end{equation}
Then, the three sets $\mathfrak{G}_1,\mathfrak{G}_2,\mathfrak{G}_3$ present a decomposition of the set of all modified $g$-vectors. Then, we can give some expressions of $\mathrm{Com}^i(\Delta(B))$.
\begin{theorem}[\Cref{thm: pointwise expression of complements}]
For each $i=1,2,3$, there are the following one-to-one correspondences:
\\
\textup{($a$)} $\varphi_i \colon \mathfrak{G}_i \rightarrow \mathrm{Com}^{i}(\Delta(B))$ given by
\begin{equation}
\varphi_i(\tilde{\mathbf{g}})=\mathcal{C}^{\circ}(\tilde{\mathbf{g}},\tilde{\mathbf{g}}-\mathfrak{g}_i).
\end{equation}
\textup{($b$)} $\rho_i \colon \mathbb{Q}_{\geq 0}  \rightarrow  \mathrm{Com}^i(\Delta(B))$ given by
\begin{equation}
\rho_i\left(\frac{b}{a}\right)=\mathcal{C}^{\circ}(\mathfrak{g}_i+a\tilde{\mathbf{v}}_{t_0k_0}+b\tilde{\mathbf{v}}_{k_0s_0},a\tilde{\mathbf{v}}_{t_0k_0}+b\tilde{\mathbf{v}}_{k_0s_0}),
\end{equation}
where $\frac{b}{a} \in \mathbb{Q}_{\geq 0}$ is a irreducible fraction, that is, $a \in \mathbb{Z}_{\geq 1}$, $b \in \mathbb{Z}_{\geq 0}$, and $\gcd(a,b)=1$.
\end{theorem}
This theorem can be summarized as \Cref{fig: all complements}. We also present a recursive expression to enumerate $\mathrm{Com}^i(\Delta(B))$. Let
\begin{equation}
F_i=\varphi_i(\tilde{\mathbf{e}}_{s_0})=\rho_i(0)=\mathcal{C}^{\circ}(\tilde{\mathbf{e}}_{s_0},\tilde{\mathbf{e}}_{k_0}-\tilde{\mathbf{e}}_{t_0}).
\end{equation}
Thanks to \Cref{thm: intro inner isomorphisms between subfans in the Markov case}, we can obtain all the elements in $\mathrm{Com}^i(\Delta(B))$ as follows.
\begin{theorem}[\Cref{thm: intro inner isomorphisms between subfans in the Markov case}]
Each element in $\mathrm{Com}^i(\Delta(B))$ is obtained by applying three linear maps $\Glinearmap_{i,1}$, $\Glinearmap_{[i]T}^{S}$, and $\Glinearmap_{[i]T}^{T}$ to $F_i$ as in \Cref{fig: recursive expression of complements}.
\end{theorem}
\subsection{Structure of this paper}
This paper is organized as follows. In \Cref{sec: tropical signs}, we recall the recursions of tropical signs derived in \cite{AC25b}, and introduce some basic notations. In \Cref{sec: fractal structure}, we prove \Cref{thm: inner isomorphisms between subfans in the Markov case} concerning about the fractal structures of $C$- and $G$-patterns. In \Cref{sec: change of basis}, we introduce some bases that are compatible with the linear maps introduced by \Cref{thm: inner isomorphisms between subfans in the Markov case}. In \Cref{sec: formulas for c g vectors}, we show some formulas to express $c$- and $g$-vectors. In \Cref{sec: combinatorial method for the G fan}, we recall the definition of $G$-fans, and explain how we draw it in the class we focus on. In \Cref{sec: upper bounds}, we introduce some upper bounds of the $G$-fan. In \Cref{sec: support of the G fan}, we prove \Cref{thm: pointwise expression of complements} and \Cref{thm: recursive expression of complements} concerning about the shape of the $G$-fan. In the last \Cref{sec: open problem}, we further discuss some important open problems arising from this work.

\section{Simplification of mutations}\label{sec: tropical signs}
In this section, we exhibit the simple mutation recursion formulas for the $B$-invariant type exchange matrices, which is a special class of general cluster-cyclic type studied in \cite{AC25b, AC26}.
\subsection{Mutations of tropical signs}
It is known that every $c$-vector is sign-coherent, that is, each $\mathbf{c}_{j}^{\mathbf{w}}$ belongs to either $\mathbb{R}_{\geq 0}^3 \setminus \{\mathbf{0}\}$ or $\mathbb{R}_{\leq 0}^3 \setminus \{\mathbf{0}\}$. Let $\varepsilon_{j}^{\mathbf{w}} \in \{\pm 1\}$ be the sign of $\mathbf{c}_{j}^{\mathbf{w}}$. Then, for the $B$-invariant type, more generally the cluster-cyclic case, the following simple recursion has already been obtained.
\begin{theorem}[{\cite[Thm.~3.4]{AC25b}}]\label{thm: recursion for tropical signs}
For any $\mathbf{w} \in \mathcal{T} \setminus \{\emptyset\}$, let $k=1,2,3$ be the last index of $\mathbf{w}$.
\\
\textup{($a$)} There exists a unique $s \in \{1,2,3\} \setminus \{k\}$ such that
\begin{equation}
\varepsilon_{s}^{\mathbf{w}} \neq \varepsilon_{k}^{\mathbf{w}},
\quad
\varepsilon_{s}^{\mathbf{w}}b_{ks}^{\mathbf{w}} < 0.
\end{equation}
\textup{($b$)} Let $s \in \{1,2,3\}\setminus\{k\}$ be the above index, and let $t \in \{1,2,3\}\setminus\{k,s\}$ be the other index. Then, we have
\begin{equation}\label{eq: recursion of tropical signs}
(\varepsilon_{k}^{{\bf w}[s]},\varepsilon_{s}^{{\bf w}[s]},\varepsilon_{t}^{{\bf w}[s]})=(-\varepsilon_{k}^{{\bf w}},-\varepsilon_{s}^{{\bf w}},\varepsilon_{t}^{{\bf w}}),
\quad
(\varepsilon_{k}^{{\bf w}[t]},\varepsilon_{s}^{{\bf w}[t]},\varepsilon_{t}^{{\bf w}[t]})=(\varepsilon_{k}^{{\bf w}},\varepsilon_{s}^{{\bf w}},-\varepsilon_{t}^{{\bf w}}).
\end{equation}
\end{theorem}
This theorem provides the recursive definition of tropical signs independent of $c$-vectors. By this recursion and the following initial conditions, all the tropical signs $\varepsilon_{j}^{\mathbf{w}}$ are uniquely determined.
\begin{equation}\label{eq: initial conditions for tropical signs}
\varepsilon_{j}^{\emptyset}=1,\quad
\varepsilon_{j}^{[i]}=\begin{cases}
-1 & \textup{if $i=j$},\\
1 & \textup{if $i \neq j$}.
\end{cases}
\end{equation}
For each $\mathbf{w} \in \mathcal{T}\setminus\{\emptyset\}$, write $k=K(\mathbf{w})$, $s=S(\mathbf{w})$, and $t=T(\mathbf{w})$, where $k$, $s$, $t$ are the ones in \Cref{thm: recursion for tropical signs}. For the later study, we give a simpler recursion of these indices. Note that the initial indices $k_0=K([i])=i$, $s_0=S([i])$, and $t_0=T([i])$ are given in \Cref{tab: List of initial indices}.
\begin{table}[htbp]
\captionsetup{skip=8pt}
\centering
\begin{tabular}{c||ccc|ccc}
$B_{t_0}$
&
\multicolumn{3}{c|}{$\left(\begin{smallmatrix}
0 & - & +\\
+ & 0 & -\\
- & + & 0
\end{smallmatrix}\right)$}
& 
\multicolumn{3}{c}{$\left(\begin{smallmatrix}
0 & + & -\\
- & 0 & +\\
+ & - & 0
\end{smallmatrix}\right)$}
\\
$i$ & $1$ & $2$ & $3$ & $1$ & $2$ & $3$
\\
\hline
$(k_0,s_0,t_0)$ 
&
$(1,3,2)$
&
$(2,1,3)$
&
$(3,2,1)$
&
$(1,2,3)$
&
$(2,3,1)$
&
$(3,2,1)$
\end{tabular}
\caption{The list of $k_0=K([i])$, $s_0=S([i])$, $t_0=T([i])$.}
\label{tab: List of initial indices}
\end{table}
\par
Let $\mathcal{M}$ be the free monoid generated by two letters $S$ and $T$ with the identity element $1_{\mathcal{M}}$. We introduce the right monoid action of $\mathcal{M}$ on $\mathcal{T} \setminus \{\emptyset\}$ by
\begin{equation}
\mathbf{w}S=\mathbf{w}[S(\mathbf{w})],
\quad
\mathbf{w}T=\mathbf{w}[T(\mathbf{w})].
\end{equation}
Fix one initial mutation direction $i=1,2,3$, and the subset $\mathcal{T}^{\geq [i]}$ is called the {\em subtree} in direction $i$. We define the subset
\begin{equation}
\mathcal{T}^{<[i]S^{\infty}}=\{[i]S^n \mid n \in \mathbb{Z}_{\geq 0}\},
\end{equation}
and call it the {\em trunk} of $\mathcal{T}^{\geq [i]}$. For each $X \in \mathcal{M}$ containing at least one letter $T$, the subset $\mathcal{T}^{\geq [i]X}$ is called a {\em branch} of $\mathcal{T}^{\geq [i]}$. In particular, $\mathcal{T}^{\geq [i]S^nT}$ is called the {\em $n$th maximal branch} of $\mathcal{T}^{\geq [i]}$. Then, depending on the trunk and branches, we obtain the following recursion.
\begin{lemma}[{\cite[Lem.~6.4]{AC25b}}]\label{lem: K S T mutation formula}
The following recurrence formulas of indices hold:
\begin{align}
(K({\bf w}S), S({\bf w}S), T({\bf w}S))&=(S({\bf w}), K({\bf w}), T({\bf w})), \label{eq: K S T mutation formula by S}
\\
(K({\bf w}T), S({\bf w}T), T({\bf w}T))&=
\begin{cases}
(T({\bf w}),S({\bf w}), K({\bf w})) & \textup{if ${\bf w}$ is in a trunk},\\
(T({\bf w}),K({\bf w}),S({\bf w})) & \textup{if ${\bf w}$ is in a branch}.
\end{cases}
\end{align}
\end{lemma}

\begin{example}
After the initial conditions in \eqref{eq: initial conditions for tropical signs} and \Cref{tab: List of initial indices}, we can calculate the tropical signs and indices $K(\mathbf{w})$, $S(\mathbf{w})$, and $T(\mathbf{w})$ based on \eqref{eq: recursion of tropical signs} and \Cref{lem: K S T mutation formula}. See \Cref{fig: example of tropical signs}.
\begin{figure}[htbp]
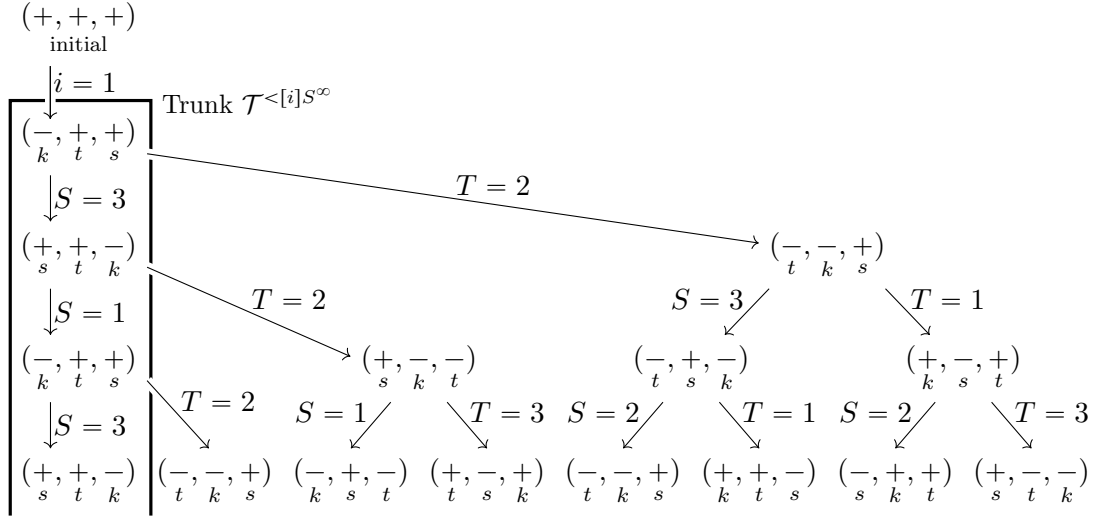

\centering
\subfile{Example}
\caption{An example of tropical signs.}
\label{fig: example of tropical signs}
\end{figure}
\end{example}

\subsection{Mutations of $c$- and $g$-vectors}
For any $\mathbf{w} \in \mathcal{T} \setminus \{\emptyset\}$ and $M=K,S,T$, we define
\begin{equation}
\tilde{\mathbf{c}}_{M}^{\mathbf{w}}=\tilde{\mathbf{c}}_{M(\mathbf{w})}^{\mathbf{w}},
\quad
\tilde{\mathbf{g}}_{M}^{\mathbf{w}}=\tilde{\mathbf{g}}_{M(\mathbf{w})}^{\mathbf{w}}.
\end{equation}
Based on the fact that $|b^{\mathbf{w}}_{ij}b^{\mathbf{w}}_{ji}|=4$ for any $\mathbf{w}$ and $i\neq j$, their mutation is given as follows.
\begin{lemma}[{\cite[Lem.~3.15,~3.16]{AC26}}]\label{lem: mutation of modified vectors}
Let $\mathbf{w} \in \mathcal{T} \setminus\{\emptyset\}$.
\\
\textup{($a$)} The $S$-mutation is given as follows:
\begin{equation}\label{eq: S mutation of c- g-vectors}
\tilde{\bf c}^{{\bf w}S}_{M}=\begin{cases}
-\tilde{\bf c}^{\bf w}_{S} & M=K,\\
\tilde{\bf c}^{\bf w}_{K}+2\tilde{\bf c}_{S}^{\bf w} & M=S,\\
\tilde{\bf c}^{\bf w}_{T} & M=T,
\end{cases}
\quad
\tilde{\bf g}^{{\bf w}S}_{M}=\begin{cases}
-\tilde{\bf g}_{S}^{\bf w}+2\tilde{\bf g}_{K}^{\bf w} & M=K,\\
\tilde{\bf g}_{K}^{\bf w} & M=S,\\
\tilde{\bf g}_{T}^{\bf w} & M=T.
\end{cases}
\end{equation}
\textup{($b$)} If $\mathbf{w}$ is in a trunk, the $T$-mutation is given as follows:
\begin{equation}\label{eq: T mutation of c- g-vectors in trunks}
\tilde{\bf c}^{{\bf w}T}_{M}=\begin{cases}
-\tilde{\bf c}^{\bf w}_{T} & M=K,\\
\tilde{\bf c}^{\bf w}_{S}+2\tilde{\bf c}_{T}^{\bf w} & M=S,\\
\tilde{\bf c}^{\bf w}_{K} & M=T,
\end{cases}
\quad
\tilde{\bf g}^{{\bf w}T}_{M}=\begin{cases}
-\tilde{\bf g}_{T}^{\bf w}+2\tilde{\bf g}_{S}^{\bf w} & M=K,\\
\tilde{\bf g}_{S}^{\bf w} & M=S,\\
\tilde{\bf g}_{K}^{\bf w} & M=T.
\end{cases}
\end{equation}
\textup{($c$)} If $\mathbf{w}$ is in a branch, the $T$-mutation is given as follows:
\begin{equation}\label{eq: T mutation of c- g-vectors in branches}
\tilde{\bf c}^{{\bf w}T}_{M}=\begin{cases}
-\tilde{\bf c}^{\bf w}_{T} & M=K,\\
\tilde{\bf c}^{\bf w}_{K}+2\tilde{\bf c}_{T}^{\bf w} & M=S,\\
\tilde{\bf c}^{\bf w}_{S} & M=T,
\end{cases}
\quad
\tilde{\bf g}^{{\bf w}T}_{M}=\begin{cases}
-\tilde{\bf g}_{T}^{\bf w}+2\tilde{\bf g}_{K}^{\bf w} & M=K,\\
\tilde{\bf g}_{K}^{\bf w} & M=S,\\
\tilde{\bf g}_{S}^{\bf w} & M=T.
\end{cases}
\end{equation}
\end{lemma}
Thus, after the first mutation $[i]$, we can obtain all modified $c$-, $g$-vectors based on the above rules. From now on, we often fix one initial mutation $i=1,2,3$, and, unless there is a risk of confusion, we always set $k_0=K([i])$, $s_0=S([i])$, and $t_0=T([i])$, which are given by \Cref{tab: List of initial indices} explicitly. Then, the first modified $c$-, $g$-vectors are given as follows.
\begin{equation}\label{eq: initial setup}
\left\{\begin{aligned}
\tilde{\bf c}_{K}^{[i]}&=-\tilde{\bf e}_{k_0},
\\
\tilde{\bf c}_{S}^{[i]}&=\tilde{\bf e}_{s_0}+2\tilde{\bf e}_{k_0},
\\
\tilde{\bf c}_{T}^{[i]}&=\tilde{\bf e}_{t_0}.
\end{aligned}\right.
\quad
\left\{\begin{aligned}
\tilde{\bf g}_{K}^{[i]}&=-\tilde{\bf e}_{k_0}+2\tilde{\bf e}_{s_0},
\\
\tilde{\bf g}_{S}^{[i]}&=\tilde{\bf e}_{s_0},
\\
\tilde{\bf g}_{T}^{[i]}&=\tilde{\bf e}_{t_0},
\end{aligned}\right.
\end{equation}

\section{Fractal structure}\label{sec: fractal structure}
Our central observation in this paper is that, depending on trunks and branches, most sub $C$-, $G$-patterns are isomorphic via linear maps. In this section, we explain this phenomenon.
\begin{definition}
Let $\mathbf{w}, \mathbf{u} \in \mathcal{T} \setminus \{\emptyset\}$. We say that $\mathbf{w}$ and $\mathbf{u}$ form an \emph{admissible pair in trunks} if both $\mathbf{w}$ and $\mathbf{u}$ belong to trunks (not necessarily the same). Similarly, we say that $\mathbf{w}$ and $\mathbf{u}$ form an \emph{admissible pair in branches} if both $\mathbf{w}$ and $\mathbf{u}$ belong to branches. In either case, $\mathbf{w}$ and $\mathbf{u}$ are simply said to be \emph{admissible}.
\end{definition}
For example, pairs of sequences $(\mathbf{w},\mathbf{u})=([1]S^4,[1]S^3),([1]S^5,[2]S^2)$ are admissible in trunks. Also, $(\mathbf{w},\mathbf{u})=([1]S^2TS^2T,[2]T^2S),([1]TS,[2]TS)$ are admissible in branches.
\par
For any $\mathbf{w} \in \mathcal{T} \setminus \{\emptyset\}$, we write
\begin{equation}
\widetilde{\mathbf{C}}^{\geq \mathbf{w}}(B)=\{(\tilde{\mathbf{c}}_1^{\mathbf{u}},\tilde{\mathbf{c}}_2^{\mathbf{u}},\tilde{\mathbf{c}}_3^{\mathbf{u}})\}_{\mathbf{u} \geq \mathbf{w}},
\quad
\widetilde{\mathbf{G}}^{\geq \mathbf{w}}(B)=\{(\tilde{\mathbf{g}}_1^{\mathbf{u}},\tilde{\mathbf{g}}_2^{\mathbf{u}},\tilde{\mathbf{g}}_3^{\mathbf{u}})\}_{\mathbf{u} \geq \mathbf{w}}.
\end{equation}
\par
The following theorem serves as a central tool in this paper.

\begin{theorem}[Fractal structure]\label{thm: inner isomorphisms between subfans in the Markov case}
Let $B$ be a skew-symmetrizable matrix in \eqref{eq: invariant type}. Let ${\bf w},{\bf u} \in \mathcal{T}\backslash\{\emptyset\}$ be an admissible pair. Then, there exist linear automorphisms $\Clinearmap_{\bf w}^{\bf u},\Glinearmap_{\bf w}^{\bf u} \in \mathrm{GL}(\mathbb{R}^3)$ on $\mathbb{R}^3$ such that
\begin{equation}\label{eq: isomorphism in the Markov quiver}
\Clinearmap_{\bf w}^{\bf u}(\tilde{\bf c}_{M}^{{\bf w}X})=\tilde{\bf c}_{M}^{{\bf u}X},
\quad
\Glinearmap_{\bf w}^{\bf u}(\tilde{\bf g}_{M}^{{\bf w}X})=\tilde{\bf g}_{M}^{{\bf u}X}
\end{equation}
for any $X \in \mathcal{M}$ and $M=K,S,T$.
\end{theorem}
Specifically, for any admissible pair $\mathbf{w}, \mathbf{u} \in \mathcal{T}$, there exist linear isomorphisms between $\widetilde{\mathbf{C}}^{\geq \mathbf{w}}(B)$ and $\widetilde{\mathbf{C}}^{\geq \mathbf{u}}(B)$, and similarly between $\widetilde{\mathbf{G}}^{\geq \mathbf{w}}(B)$ and $\widetilde{\mathbf{G}}^{\geq \mathbf{u}}(B)$. We write them as
\begin{equation}
\Clinearmap_{\mathbf{w}}^{\mathbf{u}}:\widetilde{\mathbf{C}}^{\geq \mathbf{w}}(B) \to \widetilde{\mathbf{C}}^{\geq \mathbf{u}}(B),
\quad
\Glinearmap_{\mathbf{w}}^{\mathbf{u}}:\widetilde{\mathbf{G}}^{\geq \mathbf{w}}(B) \to \widetilde{\mathbf{G}}^{\geq \mathbf{u}}(B).
\end{equation}
By considering the case of $X=1_{\mathcal{M}}$ in \eqref{eq: isomorphism in the Markov quiver}, these linear maps should satisfy
\begin{equation}\label{eq: definition of Clinear and Glinear}
\Clinearmap_{\mathbf{w}}^{\mathbf{u}}(\tilde{\mathbf{c}}_{M}^{\mathbf{w}})=\tilde{\mathbf{c}}_{M}^{\mathbf{u}},
\quad
\Glinearmap_{\mathbf{w}}^{\mathbf{u}}(\tilde{\mathbf{g}}_{M}^{\mathbf{w}})=\tilde{\mathbf{g}}_{M}^{\mathbf{u}}.
\end{equation}
Since $\{\tilde{\mathbf{c}}_{K}^{\mathbf{w}}, \tilde{\mathbf{c}}_{S}^{\mathbf{w}}, \tilde{\mathbf{c}}_{T}^{\mathbf{w}}\}$ and $\{\tilde{\mathbf{g}}_{K}^{\mathbf{w}}, \tilde{\mathbf{g}}_{S}^{\mathbf{w}}, \tilde{\mathbf{g}}_{T}^{\mathbf{w}}\}$ are bases of $\mathbb{R}^3$, these linear maps $\Clinearmap_{\mathbf{w}}^{\mathbf{u}}$ and $\Glinearmap_{\mathbf{w}}^{\mathbf{u}}$ are uniquely determined by \eqref{eq: definition of Clinear and Glinear}. In particular, it is direct that $(\Clinearmap_{\bf w}^{\bf u})^{-1}=\Clinearmap_{\bf u}^{\bf w}$ and $(\Glinearmap_{\bf w}^{\bf u})^{-1}=\Glinearmap_{\bf u}^{\bf w}$.
\begin{proof}
We show that the linear map $\Glinearmap_{\bf w}^{\bf u}$ given by \eqref{eq: definition of Clinear and Glinear} satisfies (\ref{eq: isomorphism in the Markov quiver}) by the induction on $X \in \mathcal{M}$. (We may do a similar argument for $c$-vectors.) When $X=1_{\mathcal{M}}$, by definition, the claim holds.
Suppose that the claim holds for some $X$. Then, we show that the claim holds for $XS$ and $XT$. Note that ${\bf w}X$ and ${\bf u}X$ are also admissible. Thus, they obey the same mutation rule in \eqref{eq: S mutation of c- g-vectors}, \eqref{eq: T mutation of c- g-vectors in trunks}, and \eqref{eq: T mutation of c- g-vectors in branches}. For $XS$, we may verify $\Glinearmap_{\bf w}^{\bf u}(\tilde{\bf g}_{K}^{{\bf w}XS})=\tilde{\bf g}_{K}^{{\bf u}XS}$ as follows:
\begin{equation}
\begin{aligned}
\Glinearmap_{\bf w}^{\bf u}(\tilde{\bf g}_{K}^{{\bf w}XS})
=-\Glinearmap_{\bf w}^{\bf u}(\tilde{\bf g}_{S}^{{\bf w}X})+2\Glinearmap_{\bf w}^{\bf u}(\tilde{\bf g}_{K}^{{\bf w}X})
=-\tilde{\bf g}_{S}^{{\bf u}X}+2\tilde{\bf g}_{K}^{{\bf u}X}
=\tilde{\bf g}_{K}^{{\bf u}XS},
\end{aligned}
\end{equation}
where the first and the last equalities come from \eqref{eq: S mutation of c- g-vectors}. We also have $\Glinearmap_{\bf w}^{\bf u}(\tilde{\bf g}_{S}^{{\bf w}XS})=\Glinearmap_{\bf w}^{\bf u}(\tilde{\bf g}_{K}^{{\bf w}X})=\tilde{\bf g}_{K}^{{\bf u}X}=\tilde{\bf g}_{S}^{{\bf u}XS}$ and, by the same argument, $\Glinearmap_{\bf w}^{\bf u}(\tilde{\bf g}_{T}^{{\bf w}XS})=\tilde{\bf g}_{T}^{{\bf u}XS}$. Thus, we show the claim. Similarly, we show the claim for $XT$.
\end{proof}
Let $\mathbf{w}$ be in a branch. Then, for any $X \in \mathcal{M}$, $\mathbf{w}X$ is also in a branch. Thus, due to \Cref{thm: inner isomorphisms between subfans in the Markov case}, we can find isomorphisms $\Clinearmap_{\mathbf{w}}^{\mathbf{w}X}:\widetilde{\mathbf{C}}^{\geq \mathbf{w}}(B) \to \widetilde{\mathbf{C}}^{\geq \mathbf{w}X}(B)$ and $\Glinearmap_{\mathbf{w}}^{\mathbf{w}X}:\widetilde{\mathbf{G}}^{\geq \mathbf{w}}(B) \to \widetilde{\mathbf{G}}^{\geq \mathbf{w}X}(B)$, which are inner isomorphisms. (This is a main reason we call \Cref{thm: inner isomorphisms between subfans in the Markov case} the fractal structure.) For simplicity, we introduce the following notation.
\begin{definition}
For each branch $\mathcal{T}^{\geq \mathbf{w}}$ and $X \in \mathcal{M}$, we define
\begin{equation}
\Clinearmap_{\mathbf{w}}^{X}=\Clinearmap_{\mathbf{w}}^{\mathbf{w}X}:\widetilde{\mathbf{C}}^{\geq \mathbf{w}}(B) \to \widetilde{\mathbf{C}}^{\geq \mathbf{w}X}(B),
\quad
\Glinearmap_{\mathbf{w}}^X=\Glinearmap_{\mathbf{w}}^{\mathbf{w}X}:\widetilde{\mathbf{G}}^{\geq \mathbf{w}}(B) \to \widetilde{\mathbf{G}}^{\geq \mathbf{w}X}(B).
\end{equation}
\end{definition}
These linear maps are characterized by the following equalities.
\begin{equation}\label{eq: notation X}
\Clinearmap_{\mathbf{w}}^{X}\tilde{\mathbf{c}}_{M}^{\mathbf{w}Y}=\tilde{\mathbf{c}}_{M}^{\mathbf{w}XY},
\quad
\Glinearmap_{\mathbf{w}}^{X}\tilde{\mathbf{g}}_{M}^{\mathbf{w}Y}=\tilde{\mathbf{g}}_{M}^{\mathbf{w}XY}
\quad (Y \in \mathcal{M},\ M=K,S,T).
\end{equation}

\section{Change of basis and representation matrices}\label{sec: change of basis}
In this section, we use the basis transformations in order to provide more simple recursion formulas and expressions for modified $c$-, $g$-vectors.
\subsection{Basis transformations}
The role of basis transformations is fundamental in our approach. In particular, we observe that one can choose a basis such that its mutation admits a remarkably simple expression. Hence, we define the following vectors.
\begin{align}
\tilde{\bf c}_{F}^{\bf w}&=\tilde{\bf c}_{K}^{\bf w}+\tilde{\bf c}_{S}^{\bf w}+\tilde{\bf c}_{T}^{\bf w},
&
\tilde{\bf x}_{SK}^{\bf w}&=\tilde{\bf c}_{K}^{\bf w}+\tilde{\bf c}_{S}^{\bf w},
&
\tilde{\bf x}_{TK}^{\bf w}&=-(\tilde{\bf c}_{K}^{\bf w}+\tilde{\bf c}_{T}^{\bf w}),
\label{eq: change of basis for c}
\\
\tilde{\mathbf{g}}_{F}^{\mathbf{w}}&=\tilde{\mathbf{g}}_{S}^{\mathbf{w}}+\tilde{\mathbf{g}}_{T}^{\mathbf{w}}-\tilde{\mathbf{g}}_{K}^{\mathbf{w}},
&
\tilde{\bf v}_{SK}^{\bf w}&=\tilde{\bf g}_{K}^{\bf w}-\tilde{\bf g}_{S}^{\bf w},
&
{\bf v}_{TK}^{\bf w}&=\tilde{\bf g}_{K}^{\bf w}-\tilde{\bf g}_{T}^{\bf w}.
\label{eq: change of basis for g}
\end{align}
Note that $\{\tilde{\mathbf{c}}_{F}^{\mathbf{w}},\tilde{\mathbf{x}}_{SK}^{\mathbf{w}},\tilde{\mathbf{x}}_{TK}^{\mathbf{w}}\}$ and $\{\tilde{\mathbf{g}}_{F}^{\mathbf{w}},\tilde{\mathbf{v}}_{SK}^{\mathbf{w}},\tilde{\mathbf{v}}_{TK}^{\mathbf{w}}\}$ are two bases of $\mbR^3$. We can also easily recover the modified $c$-, $g$-vectors by
\begin{align}
\tilde{\mathbf{c}}_{K}^{\mathbf{w}}&=-\tilde{\mathbf{c}}_{F}^{\mathbf{w}}+\tilde{\mathbf{x}}_{SK}^{\mathbf{w}}-\tilde{\mathbf{x}}_{TK}^{\mathbf{w}},
&
\tilde{\mathbf{c}}_{S}^{\mathbf{w}}&=\tilde{\mathbf{c}}_{F}^{\mathbf{w}}+\tilde{\mathbf{x}}_{TK}^{\mathbf{w}},
&
\tilde{\mathbf{c}}_{T}^{\mathbf{w}}&=\tilde{\mathbf{c}}_{F}^{\mathbf{w}}-\tilde{\mathbf{x}}_{SK}^{\mathbf{w}},
\label{eq: recover c vectors}
\\
\tilde{\mathbf{g}}_{K}^{\mathbf{w}}&=\tilde{\mathbf{g}}_{F}^{\mathbf{w}}+\tilde{\mathbf{v}}_{SK}^{\mathbf{w}}+\tilde{\mathbf{v}}_{TK}^{\mathbf{w}},
&
\tilde{\mathbf{g}}_{S}^{\mathbf{w}}&=\tilde{\mathbf{g}}_{F}^{\mathbf{w}}+\tilde{\mathbf{v}}_{TK}^{\mathbf{w}},
&
\tilde{\mathbf{g}}_{T}^{\mathbf{w}}&=\tilde{\mathbf{g}}_{F}^{\mathbf{w}}+\tilde{\mathbf{v}}_{SK}^{\mathbf{w}}.
\label{eq: recover g vectors}
\end{align}
Then, we give the mutation rules for these vectors based on \Cref{lem: mutation of modified vectors}.
\begin{lemma}\label{lem: mutations for eigen vectors}
Let $\mathbf{w} \in \mathcal{T}\setminus\{\emptyset\}$. The following recursion formulas hold.
\\
\textup{($a$)} The $S$-mutation rule is given by
\begin{equation}\label{eq: S mutation for eigen basis}
\left\{\begin{aligned}
\tilde{\mathbf{c}}_{F}^{\mathbf{w}S}&=\tilde{\mathbf{c}}_{F}^{\mathbf{w}},
\\
\tilde{\mathbf{x}}_{SK}^{\mathbf{w}S}&=\tilde{\mathbf{x}}_{SK}^{\mathbf{w}},
\\
\tilde{\mathbf{x}}_{TK}^{\mathbf{w}S}&=\tilde{\mathbf{x}}_{SK}^{\mathbf{w}}+\tilde{\mathbf{x}}_{TK}^{\mathbf{w}},
\end{aligned}\right.
\quad
\left\{\begin{aligned}
\tilde{\mathbf{g}}_{F}^{\mathbf{w}S}&=\tilde{\mathbf{g}}_{F}^{\mathbf{w}},
\\
\tilde{\mathbf{v}}_{SK}^{\mathbf{w}S}&=\tilde{\mathbf{v}}_{SK}^{\mathbf{w}},
\\
\tilde{\mathbf{v}}_{TK}^{\mathbf{w}S}&=\tilde{\mathbf{v}}_{SK}^{\mathbf{w}}+\tilde{\mathbf{v}}_{TK}^{\mathbf{w}}.
\end{aligned}\right.
\end{equation}
\textup{($b$)} If $\mathbf{w}$ is in a trunk, the $T$-mutation rule is given by
\begin{equation}
\left\{\begin{aligned}
\tilde{\mathbf{c}}_{F}^{\mathbf{w}T}&=\tilde{\mathbf{c}}_{F}^{\mathbf{w}},
\\
\tilde{\mathbf{x}}_{SK}^{\mathbf{w}T}&=2\tilde{\mathbf{c}}_{F}^{\mathbf{w}}-\tilde{\mathbf{x}}_{SK}^{\mathbf{w}}+\tilde{\mathbf{x}}_{TK}^{\mathbf{w}}.
\\
\tilde{\mathbf{x}}_{TK}^{\mathbf{w}T}&=2\tilde{\mathbf{c}}_{F}^{\mathbf{w}}-2\tilde{\mathbf{x}}_{SK}^{\mathbf{w}}+\tilde{\mathbf{x}}_{TK}^{\mathbf{w}},
\end{aligned}\right.
\quad
\left\{\begin{aligned}
\tilde{\mathbf{g}}_{F}^{\mathbf{w}T}&=\tilde{\mathbf{g}}_{F}^{\mathbf{w}}+2\tilde{\mathbf{v}}_{SK}^{\mathbf{w}},
\\
\tilde{\mathbf{v}}_{SK}^{\mathbf{w}T}&=\tilde{\mathbf{v}}_{TK}^{\mathbf{w}}-\tilde{\mathbf{v}}_{SK}^{\mathbf{w}},
\\
\tilde{\mathbf{v}}_{TK}^{\mathbf{w}T}&=-2\tilde{\mathbf{v}}_{SK}^{\mathbf{w}}+\tilde{\mathbf{v}}_{TK}^{\mathbf{w}}.
\end{aligned}\right.
\end{equation}
\textup{($b$)} If $\mathbf{w}$ is in a branch, the $T$-mutation rule is given by
\begin{equation}\label{eq: T mutation in branch for eigen basis}
\left\{\begin{aligned}
\tilde{\mathbf{c}}_{F}^{\mathbf{w}T}&=\tilde{\mathbf{c}}_{F}^{\mathbf{w}},
\\
\tilde{\mathbf{x}}_{SK}^{\mathbf{w}T}&=-\tilde{\mathbf{x}}_{TK}^{\mathbf{w}},
\\
\tilde{\mathbf{x}}_{TK}^{\mathbf{w}T}&=-(\tilde{\mathbf{x}}_{SK}^{\mathbf{w}}+\tilde{\mathbf{x}}_{TK}^{\mathbf{w}}),
\end{aligned}\right.
\quad
\left\{\begin{aligned}
\tilde{\mathbf{g}}_{F}^{\mathbf{w}T}&=\tilde{\mathbf{g}}_{F}^{\mathbf{w}},
\\
\tilde{\mathbf{v}}_{SK}^{\mathbf{w}T}&=\tilde{\mathbf{v}}_{TK}^{\mathbf{w}},
\\
\tilde{\mathbf{v}}_{TK}^{\mathbf{w}T}&=\tilde{\mathbf{v}}_{SK}^{\mathbf{w}}+\tilde{\mathbf{v}}_{TK}^{\mathbf{w}}.
\end{aligned}\right.
\end{equation}
\end{lemma}
\begin{proof}
All of them can be shown by \Cref{lem: mutation of modified vectors}.
\end{proof}
\begin{remark}
The sets of these vectors $[\tilde{\mathbf{c}}_{F}^{\mathbf{w}},\tilde{\mathbf{x}}_{SK}^{\mathbf{w}},\tilde{\mathbf{x}}_{TK}^{\mathbf{w}}]$ and $[\tilde{\mathbf{g}}_{F}^{\mathbf{w}},\tilde{\mathbf{v}}_{SK}^{\mathbf{w}},\tilde{\mathbf{v}}_{TK}^{\mathbf{w}}]$ are the Jordan bases of the $S$-mutations $\Clinearmap_{\mathbf{w}}^S$ and $\Glinearmap_{\mathbf{w}}^S$, respectively. One special phenomenon is that, for $T$-mutations $\Clinearmap_{\mathbf{w}}^T$ and $\Glinearmap_{\mathbf{w}}^T$ in branches, the mutation rule can also be written so simply as in \eqref{eq: T mutation in branch for eigen basis}.
\end{remark}
Based on these rules, we give the explicit expressions for some special $\mathbf{w}$.
\begin{lemma}\label{lem: c g vectors in trunks}
For any $n \in \mathbb{Z}_{\geq 0}$, we have
\begin{equation}\label{eq: c g vectors in trunks}
\left\{\begin{aligned}
\tilde{\mathbf{c}}_{F}^{[i]S^n}&=\tilde{\mathbf{e}}_{k_0} + \tilde{\mathbf{e}}_{s_0} + \tilde{\mathbf{e}}_{t_0},
\\
\tilde{\mathbf{x}}_{SK}^{[i]S^n}&=\tilde{\mathbf{e}}_{k_0}+\tilde{\mathbf{e}}_{s_0},
\\
\tilde{\mathbf{x}}_{TK}^{[i]S^n}&=(\tilde{\mathbf{e}}_{k_0}-\tilde{\mathbf{e}}_{t_0})+n(\tilde{\mathbf{e}}_{k_0}+\tilde{\mathbf{e}}_{s_0}),
\end{aligned}\right.
\quad
\left\{\begin{aligned}
\tilde{\mathbf{g}}_{F}^{[i]S^n}&=\tilde{\mathbf{e}}_{k_0} + \tilde{\mathbf{e}}_{t_0} - \tilde{\mathbf{e}}_{s_0},
\\
\tilde{\mathbf{v}}_{SK}^{[i]S^n}&=\tilde{\mathbf{e}}_{s_0}-\tilde{\mathbf{e}}_{k_0},
\\
\tilde{\mathbf{v}}_{TK}^{[i]S^n}&=(\tilde{\mathbf{e}}_{s_0}-\tilde{\mathbf{e}}_{t_0})+(n+1)(\tilde{\mathbf{e}}_{s_0}-\tilde{\mathbf{e}}_{k_0}).
\end{aligned}\right.
\end{equation}
\end{lemma}

\begin{lemma}\label{lem: c g vectors at root of maximal branches}
For any $n \in \mathbb{Z}_{\geq 0}$, we have
\begin{equation}\label{eq: root of maximal branches}
\left\{\begin{aligned}
\tilde{\mathbf{c}}_{F}^{[i]S^nT}&=\tilde{\mathbf{e}}_{k_0} + \tilde{\mathbf{e}}_{s_0} + \tilde{\mathbf{e}}_{t_0},
\\
\tilde{\mathbf{x}}_{SK}^{[i]S^nT}&=(\tilde{\mathbf{e}}_{k_0}+\tilde{\mathbf{e}}_{t_0})+(n+1)(\tilde{\mathbf{e}}_{k_0}+\tilde{\mathbf{e}}_{s_0}),
\\
\tilde{\mathbf{x}}_{TK}^{[i]S^nT}&=(\tilde{\mathbf{e}}_{k_0}+\tilde{\mathbf{e}}_{t_0})+n(\tilde{\mathbf{e}}_{k_0}+\tilde{\mathbf{e}}_{s_0}),
\end{aligned}\right.
\quad
\left\{\begin{aligned}
\tilde{\mathbf{g}}_{F}^{[i]S^nT}&=\tilde{\mathbf{e}}_{s_0}+\tilde{\mathbf{e}}_{t_0}-\tilde{\mathbf{e}}_{k_0},
\\
\tilde{\mathbf{v}}_{SK}^{[i]S^nT}&=(\tilde{\mathbf{e}}_{k_0}-\tilde{\mathbf{e}}_{t_0})+(n+1)(\tilde{\mathbf{e}}_{s_0}-\tilde{\mathbf{e}}_{k_0}),
\\
\tilde{\mathbf{v}}_{TK}^{[i]S^nT}&=(\tilde{\mathbf{e}}_{k_0}-\tilde{\mathbf{e}}_{t_0})+n(\tilde{\mathbf{e}}_{s_0}-\tilde{\mathbf{e}}_{k_0}).
\end{aligned}\right.
\end{equation}
\end{lemma}
\begin{proof}[Proof of \Cref{lem: c g vectors in trunks} and \Cref{lem: c g vectors at root of maximal branches}]
They can be shown by the induction on $n$ by using \eqref{eq: initial setup} and \Cref{lem: mutations for eigen vectors}.
\end{proof}
For later purposes, we mention that  by \eqref{eq: recover g vectors}, $\tilde{\mathbf{g}}_{M}^{[i]S^nT} (M=K,S,T)$ can be expressed as a nonnegative linear combination of $\{\tilde{\mathbf{e}}_{s_0}+\tilde{\mathbf{e}}_{t_0}-\tilde{\mathbf{e}}_{k_0}, \tilde{\mathbf{e}}_{k_0}-\tilde{\mathbf{e}}_{t_0}, \tilde{\mathbf{e}}_{s_0}-\tilde{\mathbf{e}}_{k_0}\}$.

From these lemmas, we obtain the following important observation.
\begin{lemma}\label{lem: independence of cF and gF}
We have the following claims.
\\
\textup{($a$)} For any $\mathbf{w} \in \mathcal{T}\setminus\{\emptyset\}$, the vector $\tilde{\mathbf{c}}^{\mathbf{w}}_{F}$ is independent of $\mathbf{w}$, which is given by
\begin{equation}
\tilde{\mathbf{c}}_{F}^{\mathbf{w}}=\tilde{\mathbf{e}}_{1}+\tilde{\mathbf{e}}_2+\tilde{\mathbf{e}}_3.
\end{equation}
\textup{($b$)} Let $\mathbf{w}$ be in a branch. Then, $\tilde{\mathbf{g}}_{F}^{\mathbf{w}}$ depends only on the initial mutation $i=1,2,3$, and it is given by
\begin{equation}
\tilde{\mathbf{g}}_{F}^{\mathbf{w}}=\tilde{\mathbf{e}}_{s_0}+\tilde{\mathbf{e}}_{t_0}-\tilde{\mathbf{e}}_{k_0}.
\end{equation}
\end{lemma}
Due to the importance of these vectors, we write
\begin{equation}\label{eq: fixed points of mutations}
\mathfrak{c}=\tilde{\mathbf{e}}_{1}+\tilde{\mathbf{e}}_{2}+\tilde{\mathbf{e}}_{3},
\quad
\mathfrak{g}_i=\tilde{\mathbf{e}}_{s_0}+\tilde{\mathbf{e}}_{t_0}-\tilde{\mathbf{e}}_{k_0}.
\end{equation}
Then, for any admissible pair $\mathbf{w}, \mathbf{u} \in \mathcal{T}^{\geq [i]}$ in branches with the same initial mutation direction $i=1,2,3$, by \Cref{thm: inner isomorphisms between subfans in the Markov case} and \Cref{lem: independence of cF and gF}, we have
\begin{equation}
\Clinearmap_{\mathbf{w}}^{\mathbf{u}}(\mathfrak{c})=\mathfrak{c},
\quad
\Glinearmap_{\mathbf{w}}^{\mathbf{u}}(\mathfrak{g}_i)=\mathfrak{g}_i.
\end{equation}

\subsection{Representation matrices between different branches}
Before calculating representation matrices, we mention the following fact.
\begin{lemma}\label{lem: fundamental relations for linear maps}
The following relations hold.
\\
\textup{($a$)} Let $\mathbf{w}_0,\mathbf{w}_1,\mathbf{w}_2 \in \mathcal{T}\setminus\{\emptyset\}$, and suppose that each pair is admissible. Then, we have
\begin{equation}\label{eq: chain rule}
\Clinearmap_{\mathbf{w}_1}^{\mathbf{w}_2}\Clinearmap_{\mathbf{w}_0}^{\mathbf{w}_1}=\Clinearmap_{\mathbf{w}_0}^{\mathbf{w}_2},
\quad
\Glinearmap_{\mathbf{w}_1}^{\mathbf{w}_2}\Glinearmap_{\mathbf{w}_0}^{\mathbf{w}_1}=\Glinearmap_{\mathbf{w}_0}^{\mathbf{w}_2}.
\end{equation}
\textup{($b$)} For any admissible pair $\mathbf{w},\mathbf{u} \in \mathcal{T}\setminus\{\emptyset\}$ and $X \in \mathcal{M}$, $\mathbf{w}X$ and $\mathbf{u}X$ are also admissible, and we have
\begin{equation}\label{eq: X elimination}
\Clinearmap_{\mathbf{w}X}^{\mathbf{u}X}=\Clinearmap_{\mathbf{w}}^{\mathbf{u}},
\quad
\Glinearmap_{\mathbf{w}X}^{\mathbf{u}X}=\Glinearmap_{\mathbf{w}}^{\mathbf{u}}.
\end{equation}
\end{lemma}
\begin{proof}
($a$) We prove $\Glinearmap_{\mathbf{w}_1}^{\mathbf{w}_2}\Glinearmap_{\mathbf{w}_0}^{\mathbf{w}_1}=\Glinearmap_{\mathbf{w}_0}^{\mathbf{w}_2}$ based on \Cref{thm: inner isomorphisms between subfans in the Markov case}. For each $M=K,S,T$, we have
\begin{equation}
\Glinearmap_{\mathbf{w}_1}^{\mathbf{w}_2}\Glinearmap_{\mathbf{w}_0}^{\mathbf{w}_1}(\tilde{\mathbf{g}}_{M}^{\mathbf{w}_0})=\tilde{\mathbf{g}}_{M}^{\mathbf{w}_2}=\Glinearmap_{\mathbf{w}_0}^{\mathbf{w}_2}(\tilde{\mathbf{g}}_{M}^{\mathbf{w}_0}).
\end{equation}
Since $\{\tilde{\mathbf{g}}_{K}^{\mathbf{w}_0},\tilde{\mathbf{g}}_{S}^{\mathbf{w}_0},\tilde{\mathbf{g}}_{T}^{\mathbf{w}_0}\}$ is a basis of $\mathbb{R}^3$, the above equality implies $\Glinearmap_{\mathbf{w}_1}^{\mathbf{w}_2}\Glinearmap_{\mathbf{w}_0}^{\mathbf{w}_1}=\Glinearmap_{\mathbf{w}_0}^{\mathbf{w}_2}$. The equality $\Clinearmap_{\mathbf{w}_1}^{\mathbf{w}_2}\Clinearmap_{\mathbf{w}_0}^{\mathbf{w}_1}=\Clinearmap_{\mathbf{w}_0}^{\mathbf{w}_2}$ is shown by replacing $\tilde{\mathbf{g}}_{M}^{\mathbf{w}_0}$ with $\tilde{\mathbf{c}}_{M}^{\mathbf{w}_0}$.
\\
($b$) We prove $\Glinearmap_{\mathbf{w}X}^{\mathbf{u}X}=\Glinearmap_{\mathbf{w}}^{\mathbf{u}}$. For each $M=K,S,T$, we have
\begin{equation}
\Glinearmap_{\mathbf{w}}^{\mathbf{u}}(\tilde{\mathbf{g}}_{M}^{\mathbf{w}X})=\tilde{\mathbf{g}}_{M}^{\mathbf{u}X}=\Glinearmap_{\mathbf{w}X}^{\mathbf{u}X}(\tilde{\mathbf{g}}_{M}^{\mathbf{w}X}).
\end{equation}
Thus, the statements for $\Clinearmap$ are similar and the lemma holds.
\end{proof}
For later use, we calculate some representation matrices of $\Clinearmap_{\mathbf{w}}^{\mathbf{u}}$ and $\Glinearmap_{\mathbf{w}}^{\mathbf{u}}$.
\begin{lemma}[Change of the initial mutations]\label{lem: change i to j}
For any $i,j \in \{1,2,3\}$ and $X \in \mathcal{M}$, the linear maps $\Clinearmap_{[i]X}^{[j]X}$ and $\Glinearmap_{[i]X}^{[j]X}$ are determined by the following three conditions.
\begin{equation}
\tilde{\mathbf{e}}_{K([i])} \mapsto \tilde{\mathbf{e}}_{K([j])},
\quad
\tilde{\mathbf{e}}_{S([i])} \mapsto \tilde{\mathbf{e}}_{S([j])},
\quad
\tilde{\mathbf{e}}_{T([i])} \mapsto \tilde{\mathbf{e}}_{T([j])}.
\end{equation}
\end{lemma}
\begin{proof}
By \Cref{lem: fundamental relations for linear maps}, it suffices to show the case of $X=1_{\mathcal{M}}$.
This can be shown by the definition. 
\end{proof}
Now, we focus on one subtree $\mathcal{T}^{\geq [i]}$. Set $k_0=K([i])$, $s_0=S([i])$, and $t_0=T([i])$.
The set $\mathcal{T}^{\geq [i]}$ is decomposed into the trunk $\mathcal{T}^{<[i]S^{\infty}}$ and the maximal branches $\mathcal{T}^{\geq [i]S^nT}$ with $n \in \mathbb{Z}_{\geq 0}$. The relationship among the maximal branches can be given as follows.
\begin{lemma}[Change of the maximal branches]\label{lem: insert S mutations}
For any $X \in \mathcal{M}$, the representation matrices of $\Clinearmap_{[i]X}^{[i]SX}$ and $\Glinearmap_{[i]X}^{[i]SX}$ with respect to $[\tilde{\mathbf{e}}_{k_0},\tilde{\mathbf{e}}_{s_0},\tilde{\mathbf{e}}_{t_0}]$ are
\begin{equation}
\left(\begin{matrix}
2 & -1 & 0\\
1 & 0 & 0\\
0 & 0 & 1
\end{matrix}\right),
\quad
\left(\begin{matrix}
0 & -1 & 0\\
1 & 2 & 0\\
0 & 0 & 1
\end{matrix}\right).
\end{equation}
The representation matrices of $\Clinearmap_{[i]S^mX}^{[i]S^nX}$ and $\Glinearmap_{[i]S^mX}^{[i]S^nX}$ with respect to $[\tilde{\mathbf{e}}_{k_0},\tilde{\mathbf{e}}_{s_0},\tilde{\mathbf{e}}_{t_0}]$ are
\begin{equation}
\left(\begin{matrix}
(n-m)+1 & -(n-m) & 0\\
n-m & -(n-m)+1 & 0\\
0 & 0 & 1
\end{matrix}\right),
\quad
\left(\begin{matrix}
-(n-m)+1 & -(n-m) & 0\\
n-m & (n-m)+1 & 0\\
0 & 0 & 1
\end{matrix}\right),
\end{equation}
and they depends on $n-m$, but not $n$ or $m$.
\end{lemma}
For any $i=1,2,3$ and $l \in \mathbb{Z}$, let $\Clinearmap_{i,l},\Glinearmap_{i,l} \in \mathrm{GL}(\mathbb{R}^3)$ be linear maps whose representation matrices with respect to $[\tilde{\mathbf{e}}_{k_0},\tilde{\mathbf{e}}_{s_0},\tilde{\mathbf{e}}_{t_0}]$ are given by
\begin{equation}
\left(\begin{matrix}
l+1 & -l & 0\\
l & -l+1 & 0\\
0 & 0 & 1
\end{matrix}\right),
\quad
\left(\begin{matrix}
-l+1 & -l & 0\\
l & l+1 & 0\\
0 & 0 & 1
\end{matrix}\right).
\end{equation}
\begin{example}
Fix one branch $i=1,2,3$. Then, the subtree $\mathcal{T}^{\geq [i]}$ is decomposed into the trunk $\mathcal{T}^{<[i]S^{\infty}}$ and the maximal branches $\mathcal{T}^{\geq [i]S^nT}$. Then, the relationships between two maximal branches $\mathcal{T}^{\geq [i]S^mT}$ and $\mathcal{T}^{\geq [i]S^nT}$ is given by
\begin{equation}
\Clinearmap_{[i]S^mT}^{[i]S^nT}=\Clinearmap_{i,n-m},
\quad
\Glinearmap_{[i]S^mT}^{[i]S^nT}=\Glinearmap_{i,n-m}.
\end{equation}
Let $\mathbf{w}$ and $\mathbf{u}$ be an admissible pair in branches. Let $\mathbf{w}=[i]S^nTX$ and $\mathbf{u}=[j]S^mTY$ with $n,m \in \mathbb{Z}_{\geq 0}$ and $X,Y \in \mathcal{M}$.
By \eqref{eq: chain rule}, we can express $\Clinearmap_{\mathbf{w}}^{\mathbf{u}}$ and $\Glinearmap_{\mathbf{w}}^{\mathbf{u}}$ as follows.
\begin{align}
\Clinearmap_{\mathbf{w}}^{\mathbf{u}}&=\Clinearmap^{[j]S^mTY}_{[1]S^mTY}\Clinearmap^{[1]S^mTY}_{[1]TY}\Clinearmap^{[1]TY}_{[1]T}\Clinearmap^{[1]T}_{[1]TX}\Clinearmap^{[1]TX}_{[1]S^nTX}\Clinearmap^{[1]S^nTX}_{[i]S^nTX},
\\
\Glinearmap_{\mathbf{w}}^{\mathbf{u}}&=\Glinearmap^{[j]S^mTY}_{[1]S^mTY}\Glinearmap^{[1]S^mTY}_{[1]TY}\Glinearmap^{[1]TY}_{[1]T}\Glinearmap^{[1]T}_{[1]TX}\Glinearmap^{[1]TX}_{[1]S^nTX}\Glinearmap^{[1]S^nTX}_{[i]S^nTX}.
\end{align}
By \eqref{eq: X elimination}, we can simplify each factor as follows:
\begin{align}
\Clinearmap_{\mathbf{w}}^{\mathbf{u}}&=\Clinearmap^{[j]}_{[1]}\Clinearmap^{[1]S^m}_{[1]}\Clinearmap^{[1]TY}_{[1]T}\Clinearmap^{[1]T}_{[1]TX}\Clinearmap^{[1]}_{[1]S^n}\Clinearmap^{[1]}_{[i]}=\Clinearmap_{[1]}^{[j]}\Clinearmap_{1,m}\Clinearmap_{[1]T}^{Y}(\Clinearmap_{[1]T}^{X})^{-1}\Clinearmap_{1,-n}\Clinearmap_{[i]}^{[1]},
\\
\Glinearmap_{\mathbf{w}}^{\mathbf{u}}&=\Glinearmap^{[j]}_{[1]}\Glinearmap^{[1]S^m}_{[1]}\Glinearmap^{[1]TY}_{[1]T}\Glinearmap^{[1]T}_{[1]TX}\Glinearmap^{[1]}_{[1]S^n}\Glinearmap^{[1]}_{[i]}=\Glinearmap_{[1]}^{[j]}\Glinearmap_{1,m}\Glinearmap_{[1]T}^{Y}(\Glinearmap_{[1]T}^{X})^{-1}\Glinearmap_{1,-n}\Glinearmap_{[i]}^{[1]}.
\end{align}
All the factors except $\Clinearmap_{[1]T}^{X}$, $\Clinearmap_{[1]T}^{Y}$, $\Glinearmap_{[1]T}^{X}$, and $\Glinearmap_{[1]T}^{Y}$ in the above formulas have already calculated in \Cref{lem: change i to j} and \Cref{lem: insert S mutations}. Thus, by applying \Cref{lem: fundamental relations for linear maps}, \Cref{lem: change i to j}, and \Cref{lem: insert S mutations}, the maps $\Clinearmap_{\mathbf{w}}^{\mathbf{u}}$ and $\Glinearmap_{\mathbf{w}}^{\mathbf{u}}$ for the admissible pair $\mathbf{w}$, $\mathbf{u}$ in branches are reduced to the one of the form $\Clinearmap_{[1]T}^{X}$ and $\Glinearmap_{[1]T}^{X}$ with $X \in \mathcal{M}$.
\end{example}

\subsection{Representation matrices inside one fixed branch}
In this subsection, we fix one branch $\mathcal{T}^{\geq \mathbf{w}}$.
\begin{lemma}
Suppose that $\mathbf{w}$ is in a branch. Then, for any $X,Y \in \mathcal{M}$, we have
\begin{equation}
\Clinearmap_{\mathbf{w}}^{XY}=\Clinearmap_{\mathbf{w}}^{X}\Clinearmap_{\mathbf{w}}^{Y},
\quad
\Glinearmap_{\mathbf{w}}^{XY}=\Glinearmap_{\mathbf{w}}^{X}\Glinearmap_{\mathbf{w}}^{Y}.\label{equ: XY}
\end{equation}
\end{lemma}
\begin{proof}
We prove $\Glinearmap_{\mathbf{w}}^{XY}=\Glinearmap_{\mathbf{w}}^{X}\Glinearmap_{\mathbf{w}}^{Y}$. For each $M=K,S,T$, we have
\begin{equation}
\Glinearmap_{\mathbf{w}}^{X}\Glinearmap_{\mathbf{w}}^{Y}(\tilde{\mathbf{g}}_{M}^{\mathbf{w}})=\tilde{\mathbf{g}}_{M}^{\mathbf{w}XY}=\Glinearmap_{\mathbf{w}}^{XY}(\tilde{\mathbf{g}}_{M}^{\mathbf{w}}).
\end{equation}
Thus, the claim holds.
\end{proof}
Let $\mathbf{w}=[i]S^nTX$ and $Y \in \mathcal{M}$. Then, by \eqref{eq: chain rule} and \eqref{equ: XY}, we have
\begin{equation}
\Glinearmap_{\mathbf{w}}^{Y}=\Glinearmap^{[i]S^nTXY}_{[i]S^nT}\Glinearmap^{[i]S^nT}_{[i]S^nTX}=\Glinearmap_{[i]S^nT}^{XY}(\Glinearmap_{[i]S^nT}^{X})^{-1}=\Glinearmap_{[i]S^nT}^{X}\Glinearmap_{[i]S^nT}^{Y}(\Glinearmap_{[i]S^nT}^{X})^{-1}.
\end{equation}
From this lemma, it suffices to calculate $\Clinearmap_{[i]S^nT}^{S}$, $\Clinearmap_{[i]S^nT}^{T}$, $\Glinearmap_{[i]S^nT}^{S}$, and $\Glinearmap_{[i]S^nT}^{T}$.
\begin{lemma}[Internal isomorphism of the branches]\label{lem: representation matrices in a branch}
The representation matrix of $\Glinearmap_{[i]S^nT}^{S}$ with respect to $[\tilde{\mathbf{e}}_{k_0},\tilde{\mathbf{e}}_{s_0},\tilde{\mathbf{e}}_{t_0}]$ is 
\begin{equation}
\left(\begin{smallmatrix}
-n+1 & -n & 0\\
n & n+1 & 0\\
0 & 0 & 1
\end{smallmatrix}\right)
\left(\begin{smallmatrix}
1 & 0 & 0\\
2 & 2 & 1\\
-2 & -1 & 0
\end{smallmatrix}\right)
\left(\begin{smallmatrix}
n+1 & n & 0\\
-n & -n+1 & 0\\
0 & 0 & 1
\end{smallmatrix}\right)
=
\left(\begin{smallmatrix}
-n^2-2n+1 & -n^2-n & -n\\
n^2+3n+2 & n^2+2n+2 & n+1\\
-n-2 & -n-1 & 0
\end{smallmatrix}\right).
\end{equation}
The representation matrix of $\Glinearmap_{[i]S^nT}^{T}$ with respect to $[\tilde{\mathbf{e}}_{k_0},\tilde{\mathbf{e}}_{s_0},\tilde{\mathbf{e}}_{t_0}]$ is
\begin{equation}
\left(\begin{smallmatrix}
-n+1 & -n & 0\\
n & n+1 & 0\\
0 & 0 & 1
\end{smallmatrix}\right)
\left(\begin{smallmatrix}
0 & 0 & -1\\
3 & 2 & 2\\
-2 & -1 & 0
\end{smallmatrix}\right)
\left(\begin{smallmatrix}
n+1 & n & 0\\
-n & -n+1 & 0\\
0 & 0 & 1
\end{smallmatrix}\right)
=
\left(\begin{smallmatrix}
-n^2-3n & -n^2-2n & -n-1\\
n^2+4n+3 & n^2+3n+2 & n+2\\
-n-2 & -n-1 & 0
\end{smallmatrix}\right).
\end{equation}
\end{lemma}
\begin{proof}
When $n=0$, this can be shown directly. In general, we obtain the claim by
\begin{equation}
\Glinearmap_{[i]S^nT}^{M}=\Glinearmap^{[i]S^nTM}_{[i]TM}\Glinearmap^{M}_{[i]T}\Glinearmap^{[i]T}_{[i]S^nT}=\Glinearmap_{i,n}\Glinearmap_{[i]T}^{M}\Glinearmap_{i,-n}.
\end{equation}
\end{proof}

\section{Formulas for $c$-vectors and $g$-vectors in branches}\label{sec: formulas for c g vectors}
In this section, we fix one branch $\mathcal{T}^{\geq \mathbf{w}}$, and we use the notation in \Cref{sec: change of basis}. We exhibit the combinatorial formulas for modified $c$-, $g$-vectors. In particular, for modified $g$-vectors, we reveal the connections between their generating coefficients and the Calkin-Wilf tree \cite{CW00}.
\par
By summarizing \Cref{thm: inner isomorphisms between subfans in the Markov case} and \eqref{lem: mutations for eigen vectors}, we can give the following simple recursions for modified $c$- and $g$-vectors.
\begin{theorem}\label{thm: recursive formula in a branch}
Fix a branch $\mathcal{T}^{\geq \mathbf{w}}$ with an initial mutation direction $i=1,2,3$. Then, for any $X \in \mathcal{M}$ and $M=K,S,T$, we may express
\begin{equation}\label{eq: positive expression of g-vector in the Markocv}
\begin{aligned}
\tilde{\bf c}_{M}^{{\bf w}X}&=\begin{cases}
-\mathfrak{c}+p_{K;S}^{X}\tilde{\bf x}_{SK}^{\mathbf{w}}+p_{K;T}^{X}\tilde{\bf x}_{TK}^{\bf w} & M=K,\\
\mathfrak{c}+p_{M;S}^{X}\tilde{\bf x}_{SK}^{\bf w}+p_{M;T}^{X}\tilde{\bf x}_{TK}^{\bf w} & M=S,T,
\end{cases}
\\
\tilde{\bf g}_{M}^{{\bf w}X}&=\mathfrak{g}_i+q_{M;S}^{X}\tilde{\bf v}_{SK}^{\bf w}+q_{M;T}^{X}\tilde{\bf v}_{TK}^{\bf w}.
\end{aligned}
\end{equation}
for some $(p_{M;S}^{X}, p_{M;T}^{X}) \in \mathbb{Z}^2$ and $(q_{M;S}^{X}, q_{M;T}^{X}) \in \mathbb{Z}_{\geq 0}^2$ independent of ${\bf w}$.
Moreover, for any $M=K,S,T$ and $X \in \mathcal{M}$, the coefficients $(a^{X},b^X)=(p_{M;S}^{X},p_{M;T}^{X})$ obey the following recursion independent of $M=K,S,T$:
\begin{equation}\label{eq: recursion for p}
(a^{M'X},b^{M'X})=\begin{cases}
(a^{X}+b^{X},b^X) & M'=S,\\
(-b^X,-a^X-b^X) & M'=T.
\end{cases}
\end{equation}
Also, the coefficients $(a^{X},b^X)=(q_{M;S}^{X},q_{M;T}^{X})$ obey the following recursion independent of $M=K,S,T$:
\begin{equation}\label{eq: recursion for q}
(a^{M'X},b^{M'X})=\begin{cases}
(a^{X}+b^{X},b^X) & M'=S,\\
(b^X,a^X+b^X) & M'=T.
\end{cases}
\end{equation}
\end{theorem}
\begin{proof}
We show the claim for the modified $g$-vectors by the induction on $X \in \mathcal{M}$, and we can show the other case by a similar argument. When $X=1_{\mathcal{M}}$, the claim holds by \eqref{eq: recover g vectors}. Suppose that the claim holds for some $X \in \mathcal{M}$. Then, for the $S$-mutation, we have
\begin{equation}
\begin{aligned}
\tilde{\mathbf{g}}_{M}^{\mathbf{w}SX}&=\Glinearmap_{\mathbf{w}}^{S}(\tilde{\mathbf{g}}_{M}^{\mathbf{w}X})=\Glinearmap_{\mathbf{w}}^{S}(\mathfrak{g}_i+q_{M;S}^{X}\mathbf{v}_{SK}^{\mathbf{w}}+q_{M;T}^{X}\mathbf{v}_{TK}^{\mathbf{w}})\\
\overset{\eqref{eq: S mutation for eigen basis}}&{=} \mathfrak{g}_i+q_{M;S}^{X}\mathbf{v}_{SK}^{\mathbf{w}}+q_{M;T}^{X}(\mathbf{v}_{SK}^{\mathbf{w}}+\mathbf{v}_{TK}^{\mathbf{w}})=\mathfrak{g}_i+(q_{M;S}^{X}+q_{M;T})\mathbf{v}_{SK}^{\mathbf{w}}+q_{M;T}^{X}\mathbf{v}_{TK}^{\mathbf{w}}.
\end{aligned}
\end{equation}
Also for the $T$-mutation, by the induction hypothesis, we have
\begin{equation}
\begin{aligned}
\tilde{\mathbf{g}}_{M}^{\mathbf{w}TX}&=\Glinearmap_{\mathbf{w}}^{T}(\tilde{\mathbf{g}}_{M}^{\mathbf{w}X})\overset{\eqref{eq: T mutation in branch for eigen basis}}{=}\mathfrak{g}_i+q_{M;S}^X\mathbf{v}_{TK}^{\mathbf{w}}+q_{M;T}^{X}(\mathbf{v}_{SK}^{\mathbf{w}}+\mathbf{v}_{TK}^{\mathbf{w}})
\\
&=\mathfrak{g}_i+q_{M;T}^X\mathbf{v}_{SK}^{\mathbf{w}}+(q_{M;S}^{X}+q_{M;T}^{X})\mathbf{v}_{TK}^{\mathbf{w}}.
\end{aligned}
\end{equation}
Thus, we may express each modified $g$-vector as in \eqref{eq: positive expression of g-vector in the Markocv}, and its coefficients obey the recursion \eqref{eq: recursion for q}.
\end{proof}

\begin{example}
Note that the initial conditions are given by
\begin{equation}\label{eq: initial conditions of p and q}
(p_{M;S}^{1_{\mathcal{M}}},p_{M;T}^{1_{\mathcal{M}}})=\begin{cases}
(1,-1) & M=K,\\
(0,1) & M=S,\\
(-1,0) & M=T,
\end{cases}
\quad
(q_{M;S}^{1_{\mathcal{M}}},q_{M;T}^{1_{\mathcal{M}}})=\begin{cases}
(1,1) & M=K,\\
(0,1) & M=S,\\
(1,0) & M=T.
\end{cases}
\end{equation}
Let us calculate $\mathbf{g}_{K}^{[i]S^nTS^2T^2S}$ based on \Cref{thm: recursive formula in a branch}. Let $\mathbf{w}=[i]S^nT$ and $X=S^2T^2S$. Then, following the rule \eqref{eq: recursion for q}, we obtain $(q_{K;S}^{X},q_{K;T}^{X})$ as follows:
\begin{equation}
(1,1) \xrightarrow{S} (2,1) \xrightarrow{T} (1,3) \xrightarrow{T} (3,4) \xrightarrow{S} (7,4) \xrightarrow{S} (11,4).
\end{equation}
Thus, by \eqref{eq: root of maximal branches} and \Cref{thm: recursive formula in a branch}, we obtain
\begin{equation}
\mathbf{g}_{K}^{[i]S^nTS^2T^2S}=\mathfrak{g}_i+11\tilde{\mathbf{v}}_{SK}^{[i]S^nT}+4\tilde{\mathbf{v}}_{TK}^{[i]S^nT}=(-15n+3)\tilde{\mathbf{e}}_{k_0}+(15n+12)\tilde{\mathbf{e}}_{s_0}-14\tilde{\mathbf{e}}_{t_0}.
\end{equation}
\par
By \Cref{thm: recursive formula in a branch}, we conclude that \Cref{lem: c g vectors in trunks}, \Cref{lem: c g vectors at root of maximal branches}, and the coefficients $(q^{X}_{K;S},q^{X}_{K;T})$ have enough information to recover all modified $g$-vectors. In \Cref{fig: axbx after 11}, we present examples of these coefficients $(q^X_{K;S},q^X_{K;T})$.
\begin{figure}[htbp]
\centering
\subfile{alphabeta_after_11}
\caption{$(q^{X}_{K;S},q^X_{K;T})$.}
\label{fig: axbx after 11}
\end{figure}
\par
Let $\#_T(X) \in \mathbb{Z}_{\geq 0}$ be the number of the letter $T$ appearing in $X \in \mathcal{M}$. Then, we note that $((-1)^{\#_T(X)}p_{M;S}^X, (-1)^{\#_T(X)}p_{M;T}^{X})$ obeys the rule in \eqref{eq: recursion for q}. Thus, by induction and $(p_{S;S}^S,p_{S;T}^S)=-(p_{S;S}^T,p_{S;T}^{T})=(q_{K;S}^{1_{\mathcal{M}}},q_{K;T}^{1_{\mathcal{M}}})=(1,1)$, we can relate $(p_{S;S}^{X},p_{S;T}^{X})$ to $(q_{K;S}^{X},q_{K;T}^{X})$ as follows: 
\begin{equation}\label{eq: relation between p and q}
(p_{S;S}^{XS},p_{S;T}^{XS})=(-1)^{\#_T(X)}(q_{K;S}^{X},q_{K;T}^{X}),
\quad
(p_{S;S}^{XT},p_{S;T}^{XT})=-(-1)^{\#_T(X)}(q_{K;S}^{X},q_{K;T}^{X}).
\end{equation}
\end{example}
The tree in \Cref{fig: axbx after 11} can be realized as a rearrangement of the \emph{Calkin-Wilf tree} \cite{CW00} and, equivalently, the \emph{Stern-Brocot tree}. (Note that both $(a,b)$ and $(b,a)$ appear at the same depth.) In particular, we obtain the following claim.
\begin{lemma}\label{lem: CW tree coprimeness}
In the family $\{(q_{K;S}^X,q_{K;T}^X)\}_{X \in \mathcal{M}}$, all the pairs of coprime positive numbers $(a,b) \in \mathbb{Z}_{\geq 1}^2$ appear exactly once.
\end{lemma}
\begin{proof}
We can do the same argument in \cite{CW00}.
\end{proof}
We can obtain an analogous claim for $c$-vectors.
\begin{lemma}\label{lem: CW tree lemma for p}
In the family $\{(p_{S;S}^{X},p_{S;T}^{X})\}_{X \in \mathcal{M}}$, all the elements in $(a,b) \in \mathbb{Z}^2$ satisfying either of the following conditions appear exactly once.
\begin{itemize}
    \item $(a,b)=(0,1)$.
    \item $ab \geq 1$ and $\gcd(|a|,|b|)=1$. 
\end{itemize}
\end{lemma}
\begin{proof}
This can be shown by \eqref{eq: relation between p and q} and \Cref{lem: CW tree coprimeness}.
\end{proof}

As a corollary, we obtain the following claim.
\begin{corollary}\label{cor: coprime numbers expression in a branch}
Fix one branch $\mathcal{T}^{\geq {\mathbf{w}}}$ with $\mathbf{w} \geq [i]$.
\\
\textup{($a$)} All the modified $c$-vectors appearing in $\widetilde{\mathbf{C}}^{\geq \mathbf{w}}(B)$ can be uniquely expressed as
\begin{equation}
\epsilon\mathfrak{c} + a\tilde{\mathbf{x}}_{SK}^{\mathbf{w}} + b\tilde{\mathbf{x}}_{TK}^{\mathbf{w}},
\end{equation}
where $(\epsilon; a,b) \in \{\pm \} \times (\mathbb{Z}^2 \setminus\{(0,0)\})$ satisfies either of the following:
\begin{itemize}
    \item $(\epsilon;a,b)=(-;1,-1),\ (-;0,-1),\ (-;1,0),\ (+;0,1),\ (+;-1,0)$.
    \item $\epsilon$ is arbitrary and $(a,b) \in \mathbb{Z}^2 \setminus \{(0,0)\}$ satisfies $ab \geq 1$ and $\gcd(|a|,|b|)=1$. 
\end{itemize}
\textup{($b$)} All the modified $g$-vectors appearing in $\widetilde{\mathbf{G}}^{\geq \mathbf{w}}(B)$ can be uniquely expressed as
\begin{equation}
\mathfrak{g}_i+a\tilde{\mathbf{v}}_{SK}^{\mathbf{w}}+b\tilde{\mathbf{v}}_{TK}^{\mathbf{w}},
\end{equation}
where $(a,b) \in \mathbb{Z}^2$ satisfies either of the following:
\begin{itemize}
    \item $(a,b)=(1,0),(0,1)$
    \item $a,b \geq 1$ and $\gcd(a,b)=1$.
\end{itemize}
\end{corollary}
\begin{proof}
Uniqueness is shown since $\{\mathfrak{c},\tilde{\mathbf{x}}_{SK}^{\mathbf{w}},\tilde{\mathbf{x}}_{TK}^{\mathbf{w}}\}$ and $\{\mathfrak{g}_i,\tilde{\mathbf{v}}_{SK}^{\mathbf{w}},\tilde{\mathbf{v}}_{TK}^{\mathbf{w}}\}$ are bases of $\mathbb{R}^3$.
We show the existence.
\\
($a$) By \eqref{eq: initial conditions of p and q} and some direct calculation, the cases on the first line correspond to the following list:
\begin{equation}
\begin{array}{c|ccccc}
(\epsilon; a,b) & (-;1,-1) & (-;0,-1) & (-;1,0) & (+;0,1) & (+;-1,0)
\\ 
\hline
\textup{modified $c$-vector} & \tilde{\mathbf{c}}_{K}^{\mathbf{w}} & \tilde{\mathbf{c}}_{K}^{\mathbf{w}S} & \tilde{\mathbf{c}}_{K}^{\mathbf{w}T} & \tilde{\mathbf{c}}_{S}^{\mathbf{w}} & \tilde{\mathbf{c}}_{T}^{\mathbf{w}}
\end{array}
\end{equation}
In the case on the second line, if $\epsilon=+$, we verify that all cases can be expressed as $\tilde{\mathbf{c}}_{S}^{\mathbf{w}X}=\mathfrak{c}+a\tilde{\mathbf{x}}_{SK}^{\mathbf{w}}+b\tilde{\mathbf{x}}_{TK}^{\mathbf{w}}$ thanks to \Cref{thm: recursive formula in a branch} and \Cref{lem: CW tree lemma for p}. If $\epsilon=-$, the claim can be obtained by $\tilde{\mathbf{c}}_{K}^{\mathbf{w}XS}=-\tilde{\mathbf{c}}_{S}^{\mathbf{w}X}$.
\\
($b$) By \eqref{eq: initial conditions of p and q}, the cases on the first line can be obtained by $\tilde{\mathbf{g}}_{S}^{\mathbf{w}}=\mathfrak{g}_i+\tilde{\mathbf{v}}_{TK}$ and $\tilde{\mathbf{g}}_{T}^{\mathbf{w}}=\mathfrak{g}_i+\tilde{\mathbf{v}}_{SK}$. The case on the second line can be obtained by \Cref{thm: recursive formula in a branch} and \Cref{lem: CW tree coprimeness}.
\end{proof}
For any $l,m =1,2,3$, let $\tilde{\mathbf{v}}_{lm}=\tilde{\mathbf{e}}_{m}-\tilde{\mathbf{e}}_{l}$.
The above claim provides a simple expression for $g$-vectors, which has already been shown in \cite[Thm.~3.1.5]{Cha12} and \cite[Cor.~7.3]{Rea15}.
\begin{corollary}[{\cite{Cha12}, \cite[Cor.~7.3]{Rea15}}]\label{cor: coprime number expression}
Fix one initial mutation direction $i=1,2,3$.
Then, a vector $\tilde{\mathbf{g}} \in \mathbb{R}^3$ appears in the modified $G$-pattern $\widetilde{\mathbf{G}}^{\geq [i]}(B)$ if and only if it is of the form
\begin{equation}\label{eq: coprime number expression of g vectors}
\tilde{\mathbf{g}}=\mathfrak{g}_i+a\tilde{\mathbf{v}}_{t_0k_0}+b\tilde{\mathbf{v}}_{k_0s_0},
\end{equation}
where $(a,b) \in \mathbb{Z}^2$ satisfies either of the following two conditions:
\begin{itemize}
    \item $a,b \geq 1$ and $\gcd(a,b)=1$.
    \item $(a,b)=(1,0)$ (then $\tilde{\mathbf{g}}=\tilde{\mathbf{e}}_{s_0}$) or $(a,b)=(0,-1)$ (then $\tilde{\mathbf{g}}=\tilde{\mathbf{e}}_{t_0}$).
\end{itemize}
Moreover, all the above vectors are modified $g$-vectors.
\end{corollary}
In \cite{Rea15}, this claim was shown by using the shear coordinates of allowable curves in the once-punctured torus. Here, we give an alternative proof, which only relies on the recursions in \Cref{lem: mutation of modified vectors}.
\begin{proof}
Firstly, by \eqref{eq: recover g vectors} and \Cref{lem: c g vectors in trunks}, we have 
\begin{equation}
\tilde{\mathbf{g}}_{K}^{[i]S^n}=\mathfrak{g}_i+\tilde{\mathbf{v}}_{t_0k_0}+(n+1)\tilde{\mathbf{v}}_{k_0s_0}.
\end{equation}
Thus, the above modified $g$-vectors correspond to the case of $\frac{b}{a}=1,2,\dots$ in \eqref{eq: coprime number expression of g vectors}.
According to \Cref{lem: c g vectors at root of maximal branches}, the other modified $g$-vectors $\tilde{\mathbf{g}}_{K}^{[i]S^nTX}$ can be expressed as
\begin{equation}
\mathfrak{g}_i+\alpha\tilde{\mathbf{v}}_{SK}^{[i]S^nT}+\beta\tilde{\mathbf{v}}_{TK}^{[i]S^nT}=\mathfrak{g}_i+(\alpha+\beta)\tilde{\mathbf{v}}_{t_0k_0}+((\alpha+\beta)n+\beta)\tilde{\mathbf{v}}_{k_0s_0},
\end{equation}
where $(\alpha,\beta) \in \mathbb{Z}_{\geq 1}^2$ is a pair of coprime numbers.
Set $(a,b)=(\alpha+\beta,n(\alpha+\beta)+\beta)$. As $(\alpha,\beta) \in \mathbb{Z}_{\geq 1}^2$ ranges over all pairs of coprime numbers, all such $(a,b)$ cover all pairs of coprime numbers such that $n < \frac{b}{a} < n+1$. Thus, the claim holds.
\end{proof}
Following the same method as above, we obtain the following corollary.
\begin{corollary}[{cf. \cite{Cha12}}]\label{cor: coprime number expression for c vectors}
Fix an initial mutation direction $i=1,2,3$. Set
\begin{equation}
\tilde{\mathbf{x}}_{s_0k_0}=\tilde{\mathbf{e}}_{k_0}+\tilde{\mathbf{e}}_{s_0},
\quad
\tilde{\mathbf{x}}_{t_0k_0}=\tilde{\mathbf{e}}_{k_0}+\tilde{\mathbf{e}}_{t_0}.
\end{equation}
Then, a vector $\tilde{\mathbf{c}} \in \mathbb{R}^3$ appears in the $C$-pattern $\widetilde{\mathbf{C}}^{\geq [i]}(B)$ if and only if it is of the form
\begin{equation}
\tilde{\mathbf{c}}=\epsilon\mathfrak{c}+a\tilde{\mathbf{x}}_{t_0k_0}+b\tilde{\mathbf{x}}_{s_0k_0},
\end{equation}
where $\epsilon \in \{\pm\}$ and $(a,b) \in \mathbb{Z}^2 \setminus \{(0,0)\}$ satisfy either of the following: 
\begin{itemize}
    \item $ab \geq 1$ and $\gcd(|a|,|b|)=1$ except for $(\epsilon;a,b)=(-;1,1)$.
    \item either $(+;0,-1)$, $(-;0,1)$, $(+;1,0)$, $(-;-1,0)$.
\end{itemize}
\end{corollary}
\begin{proof}
Firstly, for any $n \in \mathbb{Z}_{\geq 0}$, by \eqref{eq: recover c vectors} and \Cref{lem: c g vectors in trunks}, we have
\begin{equation}
\begin{aligned}
\tilde{\mathbf{c}}_{K}^{[i]S^n}=\mathfrak{c}-\tilde{\mathbf{x}}_{t_0k_0}-(n+1)\tilde{\mathbf{x}}_{s_0k_0},
\quad
\tilde{\mathbf{c}}_{S}^{[i]S^n}=-\mathfrak{c}+\tilde{\mathbf{x}}_{t_0k_0}+(n+2)\tilde{\mathbf{x}}_{s_0k_0},
\quad
\tilde{\mathbf{c}}_{T}^{[i]S^n}=\mathfrak{c}-\tilde{\mathbf{x}}_{s_0k_0}.
\end{aligned}
\end{equation}
By \Cref{lem: mutation of modified vectors}, we also have 
\begin{equation}
\tilde{\mathbf{c}}_{K}^{[i]S^nT}=-\mathfrak{c}+\tilde{\mathbf{x}}_{s_0k_0}
\quad
\tilde{\mathbf{c}}_{S}^{[i]S^nT}=\mathfrak{c}+\tilde{\mathbf{x}}_{t_0k_0}+n\tilde{\mathbf{x}}_{s_0k_0},
\quad
\tilde{\mathbf{c}}_{K}^{[i]S^nTS}=-\mathfrak{c}-\tilde{\mathbf{x}}_{t_0k_0}-n\tilde{\mathbf{x}}_{s_0k_0}.
\end{equation}
Thus, all the cases of $|a|=0,1$ appear.
Let us consider the modified $c$-vectors appearing in each maximal branch $\widetilde{\mathbf{C}}^{\geq [i]S^nT}(B)$. By \Cref{lem: c g vectors at root of maximal branches} and the second condition of \Cref{cor: coprime numbers expression in a branch}~$(a)$, they can be expressed as
\begin{equation}
\epsilon \mathfrak{c} + (\alpha + \beta) \tilde{\mathbf{x}}_{t_0k_0} + ((\alpha + \beta)n+\alpha) \tilde{\mathbf{x}}_{s_0k_0},
\end{equation}
where $(\alpha,\beta) \in \mathbb{Z}^2$ satisfies $\alpha \beta \geq 1$ and $\gcd(|\alpha|,|\beta|)=1$. Then, all such $(a,b)=(\alpha+\beta, n(\alpha+\beta)+\alpha) \in \mathbb{Z}^2 \setminus \{(0,0)\}$ cover the pairs of coprime numbers satisfying $n<\frac{b}{a}<n+1$. Thus, the claim holds.
\end{proof}

\section{Combinatorial method to draw the $G$-fan}\label{sec: combinatorial method for the G fan}
From now on, we derive some applications for the $G$-fan.
In this section, we shortly recall the definition of the $G$-fan and introduce a combinatorial method for drawing it for the class we focus on.
\subsection{Definition of the $G$-fan}
For any vectors $\mathbf{a}_1,\dots,\mathbf{a}_r \in \mathbb{R}^3$ and set $J \subset \{1,\dots,r\}$, we define the \emph{(polyhedral) cone} $\mathcal{C}_{J}(\mathbf{a}_1,\dots,\mathbf{a}_r)$ by
\begin{equation}
\mathcal{C}_{J}(\mathbf{a}_1,\dots,\mathbf{a}_r)=\left\{\sum_{j \in J}\lambda_j\mathbf{a}_j\ \middle|\ \lambda_j \geq 0 \right\}.
\end{equation}
By convention, we define $\mathcal{C}_{\emptyset}(\mathbf{a}_1,\dots,\mathbf{a}_r)=\{\mathbf{0}\}$.
If $J=\{1,\dots,r\}$, we omit it and simply write $\mathcal{C}(\mathbf{a}_1,\dots,\mathbf{a}_r)$.
Its \emph{relative interior} is denoted by 
\begin{equation}
\mathcal{C}^{\circ}_{J}(\mathbf{a}_1,\dots,\mathbf{a}_r)=\left\{ \sum_{j \in J} \lambda_j\mathbf{a}_j\ \middle|\ \lambda_j > 0 \right\}.
\end{equation}
Note that both sets are \emph{convex cones}; that is, for positive numbers $\lambda,\mu > 0$ and elements $\mathbf{a},\mathbf{b}$ in each set, $\lambda\mathbf{a}+\mu\mathbf{b}$ also belongs to the set.
When $\mathbf{a}_1,\dots,\mathbf{a}_r$ are linearly independent, the cone $\mathcal{C}(\mathbf{a}_1,\dots,\mathbf{a}_r)$ is said to be {\em simplicial}, and each $\mathcal{C}_{J}(\mathbf{a}_1,\dots,\mathbf{a}_r)$ is called a \emph{face} of $\mathcal{C}(\mathbf{a}_1,\dots,\mathbf{a}_r)$.
\par
For the cluster algebra theory, the following fan is one of the important objects for the structure of the cluster patterns.
\begin{definition}\label{def: G-fan}
Let $B \in \mathrm{M}_3(\mathbb{R})$ be a skew-symmetrizable matrix. We define a \emph{$G$-cone} $\mathcal{C}(G^{\mathbf{w}})$ by
\begin{equation}
\mathcal{C}(G^{\mathbf{w}})=\mathcal{C}(\tilde{\mathbf{g}}_1^{\mathbf{w}},\tilde{\mathbf{g}}_2^{\mathbf{w}},\tilde{\mathbf{g}}_3^{\mathbf{w}}).
\end{equation}
The set of all $G$-cones and their faces
\begin{equation}
\Delta(B)=\{\mathcal{C}_J(\tilde{\mathbf{g}}_1^{\mathbf{w}},\tilde{\mathbf{g}}_{2}^{\mathbf{w}},\tilde{\mathbf{g}}_{3}^{\mathbf{w}}) \mid \mathbf{w} \in \mathcal{T}, J \subset \{1,2,3\}\}
\end{equation}
is called the \emph{$G$-fan} associated with $B$.
\end{definition}
It is known that, for the ordinary (integer) cluster algebras, the $G$-fan is indeed a fan in the usual sense \cite{GHKK18}, that is, it is closed under intersection and the face relation. In general, we need the sign-coherence and some conjectures to generalize this structure, see \cite{AC25a}. However, if an initial exchange matrix is given by \eqref{eq: integer examples of B invariant type}, this generalization is easily obtained from the fact for the integer case, see \cite[Prop.~8.20]{Rea14} and \cite[Thm.~3.5]{AC25a}.
For each $\mathbf{w}_0 \in \mathcal{T}$, we often focus on the following subset of $\Delta(B)$:
\begin{equation}
\Delta^{\geq \mathbf{w}_0}(B)=\{\mathcal{C}_J(\tilde{\mathbf{g}}_1^{\mathbf{w}},\tilde{\mathbf{g}}_2^{\mathbf{w}},\tilde{\mathbf{g}}_3^{\mathbf{w}}) \mid \mathbf{w} \geq \mathbf{w}_0, J \subset \{1,2,3\}\}.
\end{equation}
This set is called a {\em sub $G$-fan} of $\Delta(B)$ after $\mathbf{w}_0$.
\subsection{Method to draw the $G$-fan}
Thanks to \Cref{cor: coprime number expression}, we can immediately prove the followig lemma.
\begin{lemma}[cf.~{\cite{Cha12, Rea15}}]\label{lem: plane for the modified g-vectors}
If $B$ is of $B$-invariant type, every modified $g$-vector $\tilde{\mathbf{g}}_{j}^{\bf w}$ ($j=1,2,3$, ${\bf w} \in \mathcal{T}$) is on the following affine plane:
\begin{equation}
H=\{x_1\tilde{\mathbf{e}}_1+x_2\tilde{\mathbf{e}}_2+x_3\tilde{\mathbf{e}}_3 \in \mathbb{R}^3 \mid x_1+x_2+x_3=1\}.
\end{equation}
In particular, the $G$-fan $\Delta(B)$ is contained in the half space
\begin{equation}
V=\{x_1\tilde{\mathbf{e}}_1+x_2\tilde{\mathbf{e}}_2+x_3\tilde{\mathbf{e}}_3 \in \mathbb{R}^3 \mid x_1+x_2+x_3 > 0 \} \cup \{\mathbf{0}\}.
\end{equation}
\end{lemma}
For the proof, we introduce some basic notations of affine planes.
\begin{itemize}
\item If $x_1+x_2+x_3 = 1$, ${\bf x}$ appears as a {\em point} on $H$. In particular, each modified standard vector $\tilde{\mathbf{e}}_i$ appears as a point on $H$.
\item If $x_1+x_2+x_3=0$, ${\bf x}$ appears as a \emph{direction vector} on $H$. In particular, for any two vectors ${\bf y},{\bf z} \in H$ which appear as points on $H$, ${\bf z}-{\bf y}$ can be seen as the direction vector from ${\bf y}$ to ${\bf z}$.
\end{itemize}
Then, we can prove \Cref{lem: plane for the modified g-vectors} with the following geometric understanding:
\begin{proof}
By \Cref{cor: coprime number expression}, this can be shown due to the following two facts:
\begin{itemize}
    \item $\mathfrak{g}_i=\tilde{\mathbf{e}}_{s_0}+\tilde{\mathbf{e}}_{s_0}-\tilde{\mathbf{e}}_{k_0}$ appears as a point on $H$.
    \item $\tilde{\mathbf{v}}_{t_0k_0}=\tilde{\mathbf{e}}_{k_0}-\tilde{\mathbf{e}}_{t_0}$ and $\tilde{\mathbf{v}}_{k_0s_0}=\tilde{\mathbf{e}}_{s_0}-\tilde{\mathbf{e}}_{k_0}$ are direction vectors on $H$.
\end{itemize}
\end{proof}
This observation provides a simple method to draw the $G$-fan as a section intersecting $H$. For $M,M' \in \{K,S,T\}$ and $\mathbf{w} \in \mathcal{T} \setminus \{\emptyset\}$ let $\tilde{\mathbf{v}}_{MM'}^{\mathbf{w}}=\tilde{\mathbf{g}}_{M'}^{\mathbf{w}}-\tilde{\mathbf{g}}_{M}^{\mathbf{w}}$ denote the direction vector from $\tilde{\mathbf{g}}_{M}^{\mathbf{w}}$ to $\tilde{\mathbf{g}}_{M'}^{\mathbf{w}}$.
Now, let us view the $S$-mutation rule $\tilde{\bf g}_{M}^{{\bf w}S}=-\tilde{\bf g}_{S}^{\bf w}+2\tilde{\bf g}_{K}^{\bf w}$ in \eqref{eq: S mutation of c- g-vectors} as follows:
\begin{equation}\label{eq: S-mutation to draw it on H}
\tilde{\bf g}_{K}^{{\bf w}S}=\tilde{\bf g}_{K}^{\bf w}+\tilde{\mathbf{v}}_{SK}^{\mathbf{w}}.
\end{equation}
On the plane $H$, by \Cref{lem: plane for the modified g-vectors}, $\tilde{\bf g}_{K}^{\bf w}$ appears as a point. Hence, the $S$-mutated vector $\tilde{\bf g}_{K}^{{\bf w}S}$ may be illustrated as a point on $H$. By doing a similar argument, we may view the $T$-mutation as follows:
\begin{equation}\label{eq: T-mutation to draw it on H}
\tilde{\bf g}_{K}^{{\bf w}T}=\begin{cases}
\tilde{\bf g}_{S}^{\bf w}+\tilde{\mathbf{v}}_{TS}^{\mathbf{w}} & \textup{if ${\bf w}$ is in a trunk},\\
\tilde{\bf g}_{K}^{\bf w}+\tilde{\mathbf{v}}_{TK}^{\mathbf{w}} & \textup{if ${\bf w}$ is in a branch}.
\end{cases}
\end{equation}
Thus, the mutation of $G$-cones may be illustrated as in Figure~\ref{fig: mutation in trunks} and Figure~\ref{fig: mutation in branches}. Note that the lines with the same color have the same length. (The red color is needed to draw $\mathcal{C}(G^{{\bf w}S})$, and the blue color is needed to draw $\mathcal{C}(G^{{\bf w}T})$.)
\begin{figure}[htbp]
\centering
\begin{minipage}[c]{0.45\linewidth}
\centering
\begin{tikzpicture}
\coordinate (A) at (0,0);
\coordinate (B) at (1,2);
\coordinate (C) at (3,0);
\fill (A) circle (0.05);
\draw ($(A)+(0.4,0.25)$) node {$T$};
\fill (B) circle (0.05);
\draw ($(B)+(0.03,-0.35)$) node {$K$};
\fill (C) circle (0.05);
\draw ($(C)+(-0.6,0.25)$) node {$S$};
\draw (A)--(B)--(C)--cycle;
\draw ($0.333*(A)+0.333*(B)+0.333*(C)$) node {${\bf w}$};

\coordinate (D) at ($2*(B)-(C)$);
\fill (D) circle (0.05);
\draw (A)--(B)--(D)--cycle;
\draw ($0.333*(A)+0.333*(B)+0.333*(D)$) node {${\bf w}S$};
\draw ($(A)+(0.075,0.5)$) node {$T$};
\draw ($(B)+(-0.3,-0.125)$) node {$S$};
\draw ($(D)+(0.5,-0.85)$) node {$K$};

\coordinate (E) at ($2*(C)-(A)$);
\fill (E) circle (0.05);
\draw (B)--(C)--(E)--cycle;
\draw ($0.333*(B)+0.333*(C)+0.333*(E)$) node {${\bf w}T$};
\draw ($(B)+(1,-0.7)$) node {$T$};
\draw ($(C)+(0.1,0.25)$) node {$S$};
\draw ($(E)+(-1.3,0.25)$) node {$K$};

\draw[red, thick] (D)--(C);
\draw[blue, thick] (A)--(E);

\foreach \x in {A,B,C,D,E}
    {
    \fill (\x) circle (0.05);
    };

\end{tikzpicture}
\caption{In a trunk.}\label{fig: mutation in trunks}
\end{minipage}
\begin{minipage}[c]{0.45\linewidth}
\centering
\begin{tikzpicture}
\coordinate (A) at (0,0);
\coordinate (B) at (1,2);
\coordinate (C) at (3,0);
\fill (A) circle (0.05);
\draw ($(A)+(0.4,0.25)$) node {$T$};
\fill (B) circle (0.05);
\draw ($(B)+(0.03,-0.35)$) node {$K$};
\fill (C) circle (0.05);
\draw ($(C)+(-0.6,0.25)$) node {$S$};
\draw (A)--(B)--(C)--cycle;
\draw ($0.333*(A)+0.333*(B)+0.333*(C)$) node {${\bf w}$};

\coordinate (D) at ($2*(B)-(C)$);
\fill (D) circle (0.05);
\draw (A)--(B)--(D)--cycle;
\draw ($0.333*(A)+0.333*(B)+0.333*(D)$) node {${\bf w}S$};
\draw ($(A)+(0.075,0.5)$) node {$T$};
\draw ($(B)+(-0.3,-0.125)$) node {$S$};
\draw ($(D)+(0.5,-0.85)$) node {$K$};

\coordinate (E) at ($2*(B)-(A)$);
\fill (E) circle (0.05);
\draw (B)--(C)--(E)--cycle;
\draw ($0.333*(B)+0.333*(C)+0.333*(E)$) node {${\bf w}T$};
\draw ($(B)+(0.3,0.1)$) node {$S$};
\draw ($(C)+(-0.4,0.7)$) node {$T$};
\draw ($(E)+(-0.1,-0.6)$) node {$K$};

\draw[red, thick] (C)--(D);
\draw[blue, thick] (A)--(E);

\foreach \x in {A,B,C,D,E}
    {
    \fill (\x) circle (0.05);
    };
\end{tikzpicture}
\caption{In a branch.}\label{fig: mutation in branches}
\end{minipage}
\end{figure}
\begin{example}
Repeating the procedures in Figure~\ref{fig: mutation in trunks} and Figure~\ref{fig: mutation in branches}, we can visualize the picture of the $G$-fan as the section of $H$ as in Figure~\ref{fig: Markov G-fan}. This figure has already been illustrated in some papers \cite{Cha12, Rea15, FG16}. Note that the region enclosed by the red lines is trunks. Later, we will show that the blue and red dashed lines are the complement of this fan, see \Cref{thm: pointwise expression of complements}.
\par
From this picture, one can observe that all dashed lines reach a black point on the red dashed line by extending it. In fact, the corresponding vectors are $\mathfrak{g}_1$ $\mathfrak{g}_2$, and $\mathfrak{g}_3$ in \eqref{eq: fixed points of mutations}, and thus the union of all $G$-cones can be easily understood as in \Cref{fig: all complements}.
\begin{figure}[htbp]
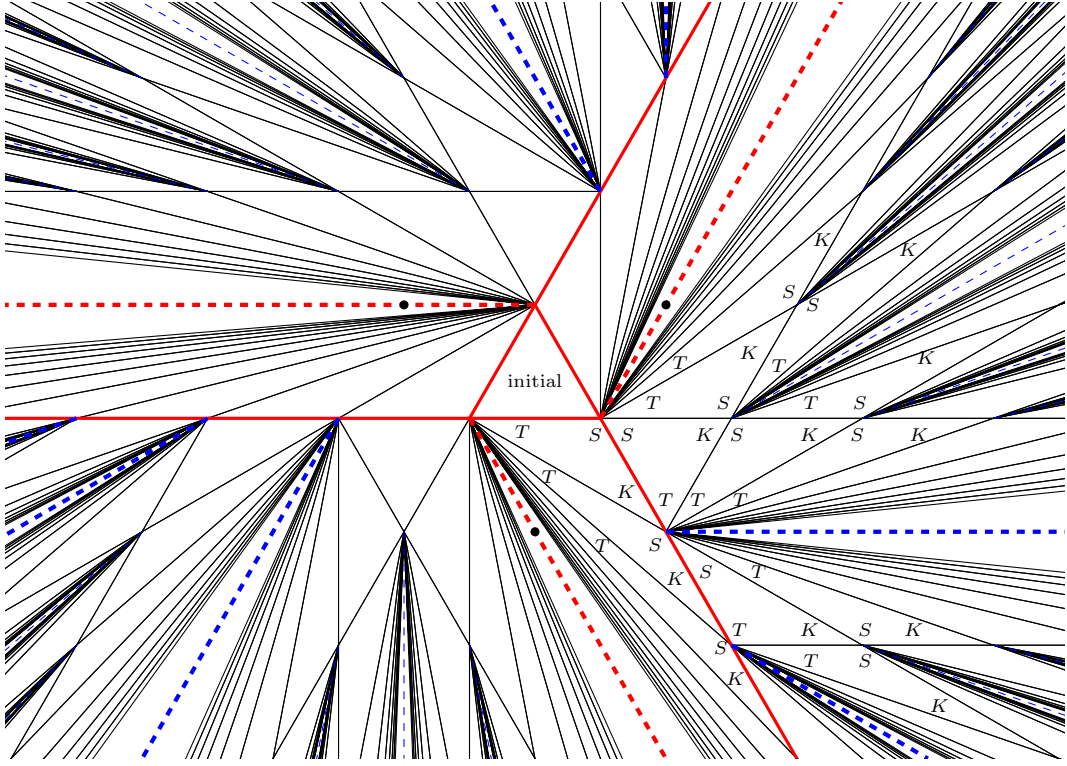

\centering
\subfile{Markov_G_fan_on_the_plane}
\caption{$G$-fan corresponding to the $B$-invariant type.}\label{fig: Markov G-fan}
\end{figure}
\end{example}
\par
Fix one $\mathbf{w} \in \mathcal{T}\setminus\{\emptyset\}$. Let us consider the union $\bigcup_{n \in \mathbb{Z}_{\geq 0}}\mathcal{C}(G^{\mathbf{w}S^n})$. Then, by repeating the process in \Cref{fig: mutation in trunks} and \Cref{fig: mutation in branches}, they can be illustrated in \Cref{fig: S mutations}. 
\begin{figure}[htbp]
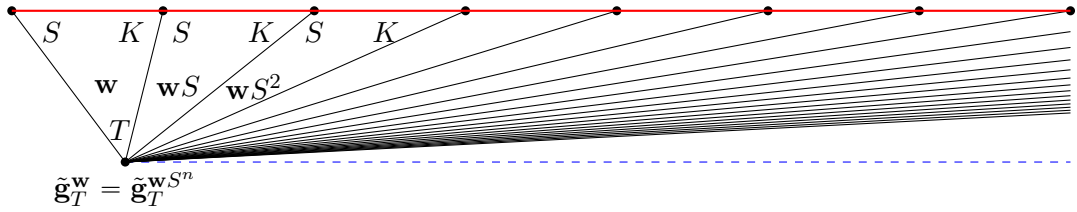

\subfile{S_mutations}
\caption{$S$-mutations.}\label{fig: S mutations}
\end{figure}
\par
Note the following two facts:
\begin{itemize}
\item $\tilde{\mathbf{g}}_{T}^{\mathbf{w}S^n}$ and $\tilde{\mathbf{v}}_{SK}^{\mathbf{w}S^n}=\tilde{\mathbf{g}}_{K}^{\mathbf{w}S^n}-\tilde{\mathbf{g}}_{S}^{\mathbf{w}S^n}$ are invariant.
\item $\tilde{\mathbf{g}}_{K}^{\mathbf{w}S^n}$ approaches the point at infinity along the half-line $\tilde{\mathbf{g}}_{K}^{\mathbf{w}}+\mathbb{R}_{\geq 0}\tilde{\mathbf{v}}_{SK}^{\mathbf{w}}$.
\end{itemize}
Thus, the asymptotic phenomenon in two dimensional faces of these cones can be expressed as follows:
\begin{itemize}
\item All $\mathcal{C}(\tilde{\mathbf{g}}_{K}^{\mathbf{w}S^n},\tilde{\mathbf{g}}_{S}^{\mathbf{w}S^n})$ with $n \in \mathbb{Z}_{\geq 0}$ compose $\mathcal{C}^{\circ}(\tilde{\mathbf{g}}_{S}^{\mathbf{w}},\tilde{\mathbf{v}}_{SK}^{\mathbf{w}}) \cup \mathcal{C}(\tilde{\mathbf{g}}_{S}^{\mathbf{w}})$, and they never reach $\mathcal{C}^{\circ}(\tilde{\mathbf{v}}_{SK}^{\mathbf{w}})$. This set is expressed as the red line in \Cref{fig: S mutations}.
\item $\mathcal{C}(\tilde{\mathbf{g}_{T}^{\mathbf{w}S^n}},\tilde{\mathbf{g}}_{K}^{\mathbf{w}S^n})$ approaches to the cone $\mathcal{C}^{\circ}(\tilde{\mathbf{g}}_{T}^{\mathbf{w}},\tilde{\mathbf{v}}_{SK}^{\mathbf{w}})\cup \mathcal{\mathcal{C}}(\tilde{\mathbf{g}
}_{T}^{\mathbf{w}})$, and they never reach to $\mathcal{C}^{\circ}(\tilde{\mathbf{g}}_{T}^{\mathbf{w}},\tilde{\mathbf{v}}_{SK}^{\mathbf{w}})$. This set $\mathcal{C}^{\circ}(\tilde{\mathbf{g}}_{T}^{\mathbf{w}},\tilde{\mathbf{v}}_{SK}^{\mathbf{w}})$ is expressed as the blue dashed line in \Cref{fig: S mutations}.
\end{itemize}

\section{Upper bounds for branches} \label{sec: upper bounds}
The contents in this section are parallel to the ones in \cite[\S~7,8]{AC26}, but the results are refined and the proof here becomes much simpler due to the specific properties of $B$-invariant type.
\subsection{Inside a branch}
We give a corollary of \Cref{thm: recursive formula in a branch} about the support of the $G$-fan. We introduce the following sets:
\begin{equation}
\mathcal{U}_{\circ}^{\mathbf{w}}
=
\mathcal{C}^{\circ}(\tilde{\mathbf{g}}_{S}^{\mathbf{w}},\tilde{\mathbf{g}}_{T}^{\mathbf{w}}, \tilde{\mathbf{v}}_{SK}^{\mathbf{w}},\tilde{\mathbf{v}}_{TK}^{\mathbf{w}}),
\quad
\mathcal{U}^{\mathbf{w}}
=
\mathcal{U}^{\mathbf{w}}_{\circ} \cup \mathcal{C}(\tilde{\mathbf{g}}_{S}^{\mathbf{w}},\tilde{\mathbf{g}}_{T}^{\mathbf{w}}).
\end{equation}
The set $\mathcal{U}^{\mathbf{w}}$ presents the upper bound of $|\Delta^{\geq \mathbf{w}}(B)|$.
\begin{lemma}\label{lem: simple upper bounds}
Let $\mathbf{w} \in \mathcal{T}^{\geq [i]}$ be in a branch. Then, we have the following inclusions.
\begin{equation}\label{eq: support upper bound in a branch}
|\Delta^{\geq \mathbf{w}}(B)|
\subset
\mathcal{U}^{\mathbf{w}}
\subset
\mathcal{C}^{\circ}(\mathfrak{g}_i,\tilde{\mathbf{v}}_{SK}^{\mathbf{w}},\tilde{\mathbf{v}}_{TK}^{\mathbf{w}}) \cup \mathcal{C}(\tilde{\mathbf{g}}_{S}^{\mathbf{w}}) \cup \mathcal{C}(\tilde{\mathbf{g}}_{T}^{\mathbf{w}}).
\end{equation}
Moreover, for any $\mathbf{u} \in \mathcal{T}^{\geq \mathbf{w}}$, the modified $g$-vector $\tilde{\mathbf{g}}_{K}^{\mathbf{u}}$ belongs to $\mathcal{U}_{\circ}^{\mathbf{w}}$, not $\mathcal{C}(\tilde{\mathbf{g}}_{S}^{\mathbf{w}},\tilde{\mathbf{g}}_{T}^{\mathbf{w}})$.
\end{lemma}
\begin{proof}
For the first inclusion, note that the set in the right hand side is a convex cone. Thus, it suffices to show that all modified $g$-vectors in $\Delta^{\geq {\mathbf{w}}}(B)$ belongs to the right hand side.
The vectors $\tilde{\mathbf{g}}_{S}^{\mathbf{w}}$ and $\tilde{\mathbf{g}}_T^{\mathbf{w}}$ belongs to $\mathcal{C}(\tilde{\mathbf{g}}_{S}^{\mathbf{w}},\tilde{\mathbf{g}}_{T}^{\mathbf{w}})$. All the other vectors can be expressed as $\tilde{\mathbf{g}}_{K}^{\mathbf{u}}$ for some $\mathbf{u} \geq \mathbf{w}$. Set $\mathbf{u}=\mathbf{w}X$. 
By \Cref{thm: recursive formula in a branch}, we have
\begin{equation}\label{eq: gK equality in a branch}
\begin{aligned}
\tilde{\mathbf{g}}_{K}^{\mathbf{u}}&=\mathfrak{g}_i+q_{K;S}^{X} \tilde{\mathbf{v}}_{SK}^{\mathbf{\mathbf{w}}}+q_{K;T}^{X}\tilde{\mathbf{v}}_{TK}^{\mathbf{w}}\\
&=\tilde{\mathbf{g}}_{K}^{\mathbf{w}}+(q_{K;S}^X-1)\tilde{\mathbf{v}}_{SK}^{\mathbf{w}}+(q_{K;T}^{X}-1)\tilde{\mathbf{v}}_{TK}^{\mathbf{w}}.
\end{aligned}
\end{equation}
Note that $q_{K;S}^{X},q_{K;T}^{X} \geq 1$. Thus, both coefficients of $\tilde{\mathbf{v}}_{SK}^{\mathbf{w}}$ and $\tilde{\mathbf{v}}_{TK}^{\mathbf{w}}$ are nonnegative. Since
\begin{equation}\label{equ: K-expression}
\tilde{\mathbf{g}}_{K}^{\mathbf{w}}=\frac{1}{2}(\tilde{\mathbf{g}}_{S}^{\mathbf{w}}+\tilde{\mathbf{g}}_{T}^{\mathbf{w}}+\tilde{\mathbf{v}}_{SK}^{\mathbf{w}}+\tilde{\mathbf{v}}_{TK}^{\mathbf{w}}) \in \mathcal{U}^{\mathbf{w}}_{\circ},
\end{equation}
we have the first inclusion and $\tilde{\mathbf{g}}_{K}^{\mathbf{u}} \in \mathcal{U}_{\circ}^{\mathbf{w}}$. The second one can be shown by
\begin{equation}
\mathcal{C}^{\circ}(\tilde{\mathbf{g}}_{S}^{\mathbf{w}},\tilde{\mathbf{g}}_{T}^{\mathbf{w}},\tilde{\mathbf{v}}_{SK}^{\mathbf{w}},\tilde{\mathbf{v}}_{TK}^{\mathbf{w}}) \cup \mathcal{C}^{\circ}(\tilde{\mathbf{g}}_{S}^{\mathbf{w}},\tilde{\mathbf{g}}_{T}^{\mathbf{w}}) \subset \mathcal{C}^{\circ}(\mathfrak{g}_i,\tilde{\mathbf{v}}_{SK}^{\mathbf{w}},\tilde{\mathbf{v}}_{TK}^{\mathbf{w}}).
\end{equation}
\end{proof} 
Note that this upper bound $\mathcal{U}^{\mathbf{w}}$ has the following decomposition.
\begin{lemma}\label{lem: direct sum decomposition of upper bounds}
We have the following direct sum decomposition.
\begin{equation}\label{eq: direct sum decomposition of local upper bounds}
\mathcal{U}^{\mathbf{w}}
=
\mathcal{C}(G^{\mathbf{w}}) \sqcup \mathcal{U}_{\circ}^{\mathbf{w}S} \sqcup \mathcal{U}_{\circ}^{\mathbf{w}T} \sqcup \mathcal{C}^{\circ}(\tilde{\mathbf{g}}_{K}^{\mathbf{w}},\tilde{\mathbf{v}}_{SK}^{\mathbf{w}}+\tilde{\mathbf{v}}_{TK}^{\mathbf{w}}).
\end{equation}
\end{lemma}
\begin{proof}
This can be shown by a direct calculation, see \Cref{fig: direct sum dicomposition}.
\end{proof}
\begin{figure}[htbp]
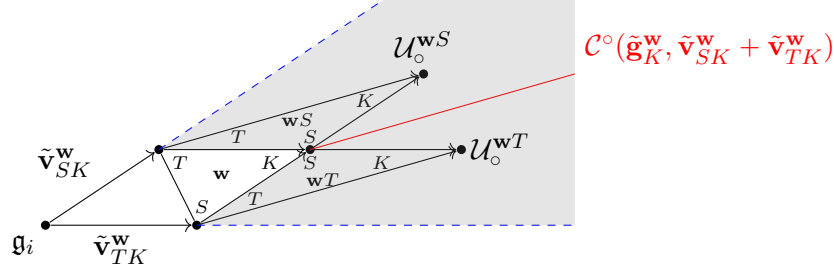

\centering
\subfile{proof_of_fractal_theorem}
\caption{Direct sum decomposition of $\mathcal{U}^{\mathbf{w}}$.}\label{fig: direct sum dicomposition}
\end{figure}
Due to the above lemma, we can improve the upper bound in \eqref{eq: support upper bound in a branch} as follows.
\begin{proposition}\label{prop: detailed upper bounds}
We have
\begin{equation}\label{eq: local upper bound theorem}
|\Delta^{\geq \mathbf{w}}(B)| \subset \mathcal{C}(G^{\mathbf{w}}) \sqcup \mathcal{U}_{\circ}^{\mathbf{w}S} \sqcup \mathcal{U}_{\circ}^{\mathbf{w}T}.
\end{equation}
\end{proposition}
\begin{proof}
This follows from $|\Delta^{\geq \mathbf{w}}(B)| = \mathcal{C}(G^{\mathbf{w}}) \cup |\Delta^{\geq \mathbf{w}S}(B)|^{\circ} \cup |\Delta^{\geq \mathbf{w}T}(B)|^{\circ}$ and \Cref{lem: simple upper bounds}.
\end{proof}
\subsection{Among branches}
For each $i=1,2,3$, $k_0=K([i]), s_0=S([i]), t_0=T([i])$, let
\begin{equation}
\mathfrak{D}_i=\mathcal{C}(\tilde{\mathbf{e}}_{s_0},\tilde{\mathbf{v}}_{t_0k_0},\tilde{\mathbf{v}}_{k_0s_0}) \cap V.
\end{equation}
This set can be illustrated as the region bounded by the red lines in \Cref{fig: maximal branches}.
As \Cref{fig: maximal branches} shows, we can directly prove that
\begin{equation}
\mathcal{U}^{[i]S^nT} \subset \mathfrak{D}_i
\end{equation}
for any $n \in \mathbb{Z}_{\geq 0}$. Moreover, as \Cref{fig: Markov G-fan} shows, for any $i,j=1,2,3$, if $i \neq j$, we have
\begin{equation}
\mathfrak{D}_i \cap \mathfrak{D}_j = \{\mathbf{0}\}.
\end{equation}
Based on this fact, we prove the following facts.
\begin{lemma}\label{lem: separateness among maximal branches}
Fix one initial mutation direction $i=1,2,3$.
\\
\textup{($a$)} For any $\mathbf{w} \in \mathcal{T}^{\geq [i]}$, $\tilde{\mathbf{g}}_{K}^{\mathbf{w}}$ belongs to $\mathfrak{D}_i$. Moreover, $\tilde{\mathbf{g}}_{K}^{\mathbf{w}}$ belongs to the interior $\mathfrak{D}_i^{\circ}$ if and only if $\mathbf{w}$ is in a branch.
\\
\textup{($b$)} For any $m,n \in \mathbb{Z}_{\geq 0}$ with $m>n$, we have
\begin{equation}
|\mathcal{U}^{[i]S^mT}| \cap |\mathcal{U}^{[i]S^nT}|=\begin{cases}
\mathcal{C}(\tilde{\mathbf{g}}_{K}^{[i]S^n}) & \textup{if $m=n+1$},\\
\emptyset & \textup{if $m \geq n+2$}.
\end{cases}
\end{equation}
\end{lemma}
\begin{proof}
($a$) By \eqref{eq: recover g vectors} and \eqref{eq: root of maximal branches}, we have $\tilde{\mathbf{g}}_{K}^{[i]S^n}=\tilde{\mathbf{e}}_{s_0}+n\tilde{\mathbf{v}}_{k_0s_0} \in \mathcal{C}(\tilde{\mathbf{e}}_{s_0},\tilde{\mathbf{v}}_{k_0s_0}) \cap V \subset \partial\mathfrak{D}_i$, where $\partial \mathfrak{D}_i$ is the boundary of $\mathfrak{D}_i$. If $\mathbf{w}$ is in a branch, we have $\tilde{\mathbf{g}}_{K}^{\mathbf{w}} \in \mathfrak{D}_i^{\circ}$ by \Cref{lem: simple upper bounds}.
\\
($b$) We use the right most upper bound in \eqref{eq: support upper bound in a branch}. By \Cref{lem: c g vectors at root of maximal branches}, we have $\tilde{\mathbf{v}}_{TK}^{[i]S^{n+1}T}=\tilde{\mathbf{v}}_{SK}^{[i]S^nT}$ and $\tilde{\mathbf{v}}_{SK}^{[i]S^{n+1}T}=\tilde{\mathbf{v}}_{SK}^{[i]S^nT}+(\tilde{\mathbf{e}}_{s_0}-\tilde{\mathbf{e}}_{k_0})$. Thus, $\mathcal{C}^{\circ}(\mathfrak{g}_i,\tilde{\mathbf{v}}_{SK}^{[i]S^nT},\tilde{\mathbf{v}}_{TK}^{[i]S^nT})$ can be illustrated in \Cref{fig: maximal branches}, and they do not have any intersection each other. Thus, if the intersection exists between two maximal branches, it is on $\mathcal{C}(\tilde{\mathbf{g}}_{S}^{[i]S^nT})=\mathcal{C}(\tilde{\mathbf{g}}_{S}^{[i]S^n})=\mathcal{C}(\tilde{\mathbf{g}}_{K}^{[i]S^{n-1}})$ or $\mathcal{C}(\tilde{\mathbf{g}}_{T}^{[i]S^nT})=\mathcal{C}(\tilde{\mathbf{g}}_{K}^{[i]S^n})$. Thus, the claim holds.
\begin{figure}[htbp]
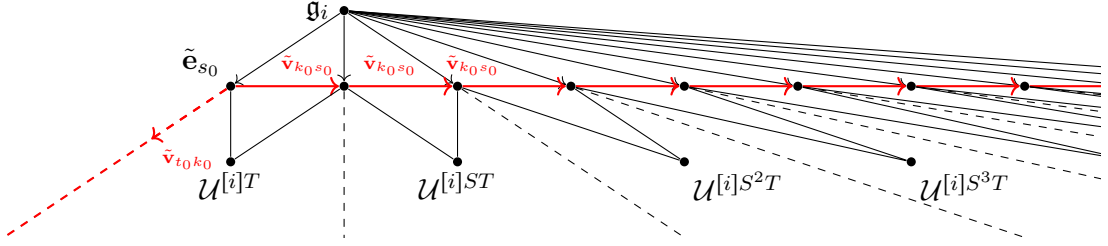

\centering
\subfile{proof_of_pointwise_complement_expressions}
\caption{Maximal branches. Here, we set $\tilde{\mathbf{v}}_{lm}=\tilde{\mathbf{e}}_{m}-\tilde{\mathbf{e}}_{l}$.} \label{fig: maximal branches}
\end{figure}
\end{proof}
As a corollary in this section, we obtain the following facts.
\begin{proposition}\label{prop: representative of g-vectors}
Every modified $g$-vector can be expressed as $\tilde{\mathbf{e}}_j$ ($j=1,2,3$) or $\tilde{\mathbf{g}}_{K}^{\mathbf{w}}$ ($\mathbf{w} \neq \emptyset$) uniquely.
\end{proposition}
\begin{proof}
The existence of this expression is obvious by \Cref{lem: mutation of modified vectors}. We show the uniqueness. Note that $\mathfrak{D}_i \cap \mathfrak{D}_j = \{\mathbf{0}\}$ when $i \neq j$. Thus, we can fix an initial mutation direction $i=1,2,3$. Moreover, as \Cref{fig: maximal branches} shows, if the coincidence happens, it should happen in the same maximal branch. Thus, for a given maximal branch $\mathcal{T}^{\geq [i]S^nT}$, it suffices to show that $\tilde{\mathbf{g}}_{K}^{\mathbf{w}} \neq \tilde{\mathbf{g}}_{K}^{\mathbf{u}}$ when $\mathbf{w} \neq \mathbf{u} \in \mathcal{T}^{\geq [i]S^nT}$. This can be shown by \Cref{prop: detailed upper bounds}.
\end{proof}
\section{Support of the $G$-fan} \label{sec: support of the G fan}
In this section, we investigate the support of the $G$-fan. By \Cref{lem: plane for the modified g-vectors}, the $G$-fan is contained in $V$.
We focus on its complement. Let $\mathrm{Com}(\Delta(B))$ be the set of all connected components in $V \setminus |\Delta(B)|$. As \Cref{fig: Markov G-fan} shows, each complement is a subset of $\mathfrak{D}_i$. Let $\mathrm{Com}^i(\Delta(B))$ be the set of all complements included in $\mathfrak{D}_i$. Then, we have the direct sum decomposition
\begin{equation}
\mathrm{Com}(\Delta(B))=\mathrm{Com}^1(\Delta(B)) \sqcup \mathrm{Com}^2(\Delta(B)) \sqcup \mathrm{Com}^3(\Delta(B)).
\end{equation}
In the following, we give some expressions of these complements.
\subsection{Pointwise expression}
Let $\mathfrak{G}$ be the set of all modified $g$-vectors. By \Cref{lem: separateness among maximal branches}, this set can be decomposed into the direct sum $\mathfrak{G}=\mathfrak{G}_1 \sqcup \mathfrak{G}_2 \sqcup \mathfrak{G}_3$, where
\begin{equation}
\mathfrak{G}_i=\mathfrak{G} \cap \mathfrak{D}_i = \{\mathbf{g}_{K}^{\mathbf{w}} \mid \mathbf{w} \in \mathcal{T}^{\geq [i]}\} \cup \{\tilde{\mathbf{e}}_{s_0}\}.
\end{equation}
Note that $s_0=S([i])$ depends on the initial mutation direction $i$. Let $\tilde{\mathbf{v}}_{lm}=\tilde{\mathbf{e}}_{m}-\tilde{\mathbf{e}}_{l}$ for $l,m=1,2,3$. By \Cref{cor: coprime number expression}, each component can also be expressed as
\begin{equation}\label{eq: coprime expression of Gi}
\mathfrak{G}_i=\{\mathfrak{g}_i+a\tilde{\mathbf{v}}_{t_0k_0}+b\tilde{\mathbf{v}}_{k_0s_0} \mid (a,b) \in \mathbb{Z}_{\geq 1}^2, \gcd(a,b)=1\} \cup \{\tilde{\mathbf{e}}_{s_0}\}
\end{equation}
Then, we can assign the modified $g$-vectors to the connected components of the complement.
\begin{theorem}[cf.~{\cite[\S~2.2]{FG16}}]\label{thm: pointwise expression of complements}
For each $i=1,2,3$, there are the following one-to-one correspondences:
\\
\textup{($a$)} $\varphi_i \colon \mathfrak{G}_i \rightarrow \mathrm{Com}^{i}(\Delta(B))$ given by
\begin{equation}
\varphi_i(\tilde{\mathbf{g}})=\mathcal{C}^{\circ}(\tilde{\mathbf{g}},\tilde{\mathbf{g}}-\mathfrak{g}_i).
\end{equation}
\textup{($b$)} $\rho_i \colon \mathbb{Q}_{\geq 0}  \rightarrow  \mathrm{Com}^i(\Delta(B))$ given by
\begin{equation}
\rho_i\left(\frac{b}{a}\right)=\mathcal{C}^{\circ}(\mathfrak{g}_i+a\tilde{\mathbf{v}}_{t_0k_0}+b\tilde{\mathbf{v}}_{k_0s_0},a\tilde{\mathbf{v}}_{t_0k_0}+b\tilde{\mathbf{v}}_{k_0s_0}),
\end{equation}
where $\frac{b}{a} \in \mathbb{Q}_{\geq 0}$ is a irreducible fraction, that is, $a \in \mathbb{Z}_{\geq 1}$, $b \in \mathbb{Z}_{\geq 0}$, and $\gcd(a,b)=1$.
\end{theorem}
In \cite{FG16}, they mentioned that each complement can be depicted as a ray on $H$. Here, we give more explicit expressions.
\begin{proof}
By \eqref{eq: coprime expression of Gi}, the claim ($a$) implies ($b$). Thus, we focus on proving ($a$).
For a given modified $g$-vector $\tilde{\mathbf{g}} \in \mathfrak{G}_i$, let us consider the following three cases.
\\
(1) Let $\tilde{\mathbf{g}}=\tilde{\mathbf{e}}_{s_0}$. Then, as \Cref{fig: maximal branches} shows, $\varphi_i(\tilde{\mathbf{e}}_{s_0})=\mathcal{C}^{\circ}(\tilde{\mathbf{e}}_{s_0}, \tilde{\mathbf{e}}_{t_0}-\tilde{\mathbf{e}}_{k_0})$ is in the boundary of $\mathfrak{D}_i$. Thus, this is a part of the complement. Moreover, every connected set strictly including $\varphi(\tilde{\mathbf{e}}_{s_0})$ intersects $\bigcup_{n}\mathcal{C}(G^{[i]T^2S^n})$ or $\bigcup_{n}\mathcal{C}(G^{[t_0]S^n})$. Thus, this is certainly a connected component of $V \setminus |\Delta(B)|$.
\\
(2) Let $\tilde{\mathbf{g}}=\tilde{\mathbf{g}}_{K}^{[i]S^n}$. As \Cref{fig: maximal branches} shows, $\varphi(\tilde{\mathbf{g}}_{K}^{[i]S^n})$ is in the boundary of both $\mathcal{U}^{[i]S^nT}$ and $\mathcal{U}^{[i]S^{n+1}T}$. Thus, this is a part of the complement. Moreover, every connected set strictly including $\varphi(\tilde{\mathbf{g}}_{K}^{[i]S^n})$ intersects $\bigcup_m \mathcal{C}(G^{[i]S^nTS^m})$ or $\bigcup_{m}\mathcal{C}(G^{[i]S^{n+1}T^2S^m})$. Thus, this is a connected component of $V \setminus |\Delta(B)|$.
\\
(3) Let $\tilde{\mathbf{g}}=\tilde{\mathbf{g}}_{K}^{\mathbf{w}}$, where $\mathbf{w}$ is in a branch. Then, by \Cref{lem: direct sum decomposition of upper bounds} and \Cref{prop: detailed upper bounds}, $\varphi(\tilde{\mathbf{g}}_{K}^{\mathbf{w}})$ is a part of the complement. See \Cref{fig: direct sum dicomposition}. Moreover, every connected set strictly including $\varphi(\tilde{\mathbf{g}}_{K}^{\mathbf{w}})$ intersects $\bigcup_m(G^{\mathbf{w}STS^m})$ or $\bigcup_m(G^{\mathbf{w}T^2S^m})$. Thus, this is a connected component of $V \setminus |\Delta(B)|$.
\end{proof}
\begin{example}
Thanks to \Cref{thm: pointwise expression of complements}, all the complements in $\mathfrak{D}_i$ can be illustrated as the dashed lines in \Cref{fig: all complements}. To make \Cref{thm: pointwise expression of complements} more clear, we draw the $2$-dimensional lattice 
$\mathfrak{g}_i+\mathbb{Z}\tilde{\mathbf{v}}_{t_0k_0} \oplus \mathbb{Z}\tilde{\mathbf{v}}_{k_0s_0} \subset H$ in this picture.
The red points express the modified $g$-vectors, which correspond to the primitive vectors in this lattice. We can assign a rational numbers $\frac{b}{a}$ in \Cref{cor: coprime number expression} to each complement. This assigned rational number coincides with the slope of each line.
\begin{figure}[htbp]
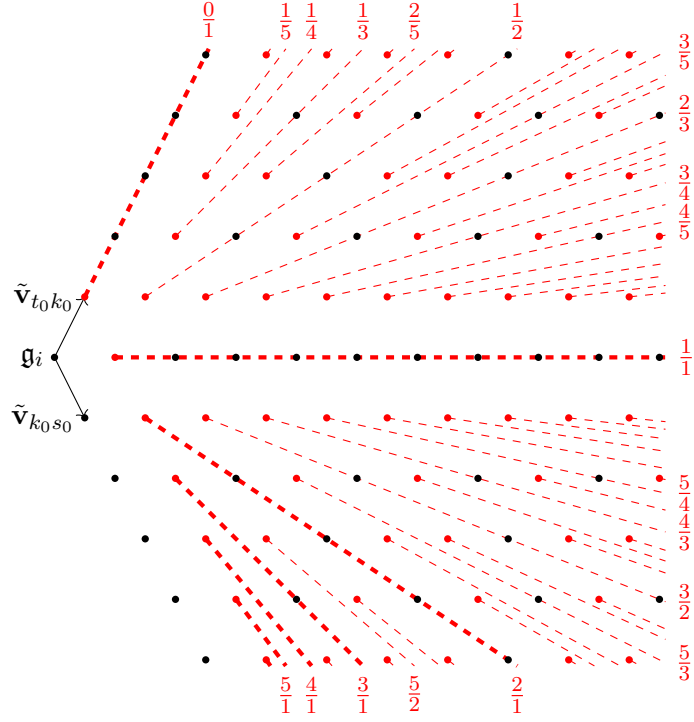

\centering
\subfile{complements.tex}
\caption{Complement in $\mathfrak{D}_i$.}
\label{fig: all complements}
\end{figure}
\end{example}
\subsection{Recursive expression}
Let $\varphi:\mathfrak{G} \to \mathrm{Com}(\Delta(B))$ be the map defined by $\varphi|_{\mathfrak{G}_i}=\varphi_i$. To simplify the notation, we also write $\varphi(\mathbf{w})=\varphi(\tilde{\mathbf{g}}_{K}^{\mathbf{w}})$ for any $\mathbf{w} \in \mathcal{T} \setminus\{\emptyset\}$. Then, by \Cref{prop: representative of g-vectors}, except for the three complements $\varphi(\tilde{\mathbf{e}}_1)$, $\varphi(\tilde{\mathbf{e}}_2)$, and $\varphi(\tilde{\mathbf{e}}_3)$, we can express all the elements in $\mathrm{Com}(\Delta(B))$ as $\varphi(\mathbf{w})$.
For each $\mathbf{w} \in \mathcal{T} \setminus \{\emptyset\}$, we define
\begin{equation}
\mathrm{Com}^{\geq \mathbf{w}}(\Delta(B))=\{\varphi(\mathbf{u}) \mid \mathbf{u} \geq \mathbf{w}\}.
\end{equation}
For any multiplicative submonoid $\Gamma \subset \mathrm{GL}(\mathbb{R}^3)$ and any subset $A \subset \mathbb{R}^3$, we can define its orbit
\begin{equation}
\Gamma [A] = \{\gamma(A) \mid \gamma \in \Gamma\}.
\end{equation}
Firstly, we give an recursive expression of $\mathrm{Com}^{\geq {\mathbf{w}}}(\Delta(B))$. Define $\Gamma_{\mathbf{w}}=\langle \Glinearmap_{\mathbf{w}}^{S}, \Glinearmap_{\mathbf{w}}^{T}\rangle_{\mathrm{mono}}$ be the multiplicative submonoid of $\mathrm{GL}(\mathbb{R}^3)$ generated by $\Glinearmap_{\mathbf{w}}^{S}$ and $\Glinearmap_{\mathbf{w}}^{T}$.
\begin{lemma}\label{lem: recursive expression in a branch}
Let $\Delta^{\geq \mathbf{w}}(B)$ be a branch. For any $X \in \mathcal{M}$ and $M=S,T$, we have
\begin{equation}
\Glinearmap_{\mathbf{w}}^{M}(\varphi(\mathbf{w}X))=\varphi(\mathbf{w}MX).
\end{equation}
In particular, the following relation holds.
\begin{equation}
\mathrm{Com}^{\geq \mathbf{w}}(\Delta(B)) = \Gamma_{\mathbf{w}}[\varphi(\mathbf{w})].
\end{equation}
\end{lemma}
\begin{proof}
Let $i=1,2,3$ be the initial mutation direction of $\mathbf{w}$. Then, the equality $\Glinearmap_{\mathbf{w}}^{[i]}(\varphi(\mathbf{w}X))=\varphi(\mathbf{w}MX)$ follows from $\Glinearmap_{\mathbf{w}}^M(\mathfrak{g}_i)=\mathfrak{g}_i$ and $\Glinearmap_{\mathbf{w}}^M(\tilde{\mathbf{g}}_{K}^{\mathbf{w}X})=\tilde{\mathbf{g}}_{K}^{\mathbf{w}MX}$.
\end{proof}
Recall that $\mathrm{Com}^{i}(\Delta(B))=\{\varphi(\mathbf{w}) \mid \mathbf{w} \geq [i]\} \cup \{\varphi(\tilde{\mathbf{e}}_{s_0})\}$. Set 
\begin{equation}
F_i=\varphi(\tilde{\mathbf{e}}_{s_0})=\rho_i(0)=\mathcal{C}^{\circ}(\tilde{\mathbf{e}}_{s_0}, \tilde{\mathbf{e}}_{k_0}-\tilde{\mathbf{e}}_{t_0}).
\end{equation}
Then, we obtain all complements in $\mathfrak{D}_i$ by applying the linear maps to $F_i$ as follows.
\begin{theorem}\label{thm: recursive expression of complements}
For any $n \in \mathbb{Z}_{\geq 1}$, we have the following relations:
\begin{align}
(\Glinearmap_{i,1})^n(F_i)=\varphi([i]S^{n-1}),
\quad
\Glinearmap_{[i]T}^{T}(F_i)=\varphi([i]T),
\quad
\Glinearmap_{[i]S^{n}T}^{T}(\varphi([i]S^{n-1}))=\varphi([i]S^{n}T).
\end{align}
For any $\mathbf{w} \in \mathcal{T}^{\geq [i]}$, we obtain $\varphi(\mathbf{w})$ by applying $\Glinearmap_{i,1}$, $\Glinearmap_{[i]T}^{S}$, and $\Glinearmap_{[i]T}^{T}$ to $F_i$ as follows:
\\
\textup{($a$)} If $\mathbf{w}=[i]S^n$, we have
\begin{equation}
\varphi(\mathbf{w})=(\Glinearmap_{i,1})^n(F_i).
\end{equation}
\textup{($b$)} If $\mathbf{w}=[i]S^nTM_1M_2\dots M_r$ with $M_i=S,T$, we have
\begin{equation}
\varphi(\mathbf{w})=(\Glinearmap_{i,1})^n\Glinearmap_{[i]T}^{M_1}\cdots\Glinearmap_{[i]T}^{M_r}\Glinearmap_{[i]T}^{T}(F_i).
\end{equation}
\end{theorem}
\begin{proof}
The first equality $(\Glinearmap_{i,1})^{n}(F_i)=\varphi([i]S^{n-1})$ follows from $\tilde{\mathbf{e}}_{s_0}=\tilde{\mathbf{g}}_{S}^{[i]}$, $(\Glinearmap_{i,1})^n=\Glinearmap_{[i]}^{[i]S^n}$, and
\begin{equation}
(\Glinearmap_{i,1})^n(\tilde{\mathbf{e}}_{s_0})=\tilde{\mathbf{g}}_{S}^{[i]S^n}=\tilde{\mathbf{g}}_{K}^{[i]S^{n-1}}.
\end{equation}
This equality implies ($a$).
The second equality follows from $\tilde{\mathbf{e}}_{s_0}=\tilde{\mathbf{g}}_{S}^{[i]T}$ and $\Glinearmap_{[i]T}^T(\tilde{\mathbf{e}}_{s_0})=\tilde{\mathbf{g}}_{S}^{[i]T^2}=\tilde{\mathbf{g}}_{K}^{[i]T}$. The third equality also follows from $\tilde{\mathbf{g}}_{K}^{[i]S^{n-1}}=\tilde{\mathbf{g}}_{S}^{[i]S^nT}$ and $\Glinearmap_{[i]S^{n}T}^T(\tilde{\mathbf{g}}_{K}^{[i]S^{n-1}})=\tilde{\mathbf{g}}_{S}^{[i]S^nT^2}=\tilde{\mathbf{g}}_{K}^{[i]S^nT}$. These three equalities and \Cref{lem: recursive expression in a branch} imply ($b$).
\end{proof}
\begin{example}
Let us consider the case where the initial exchange matrix is
\begin{equation}
B=\left(\begin{matrix}
0 & -2 & 2\\
2 & 0 & -2\\
-2 & 2 & 0
\end{matrix}\right),
\end{equation} which is skew-symmetric. Hence, all the modified $g$-vectors are the same as the ordinary $g$-vectors.
Set $i=1$. Then, the initial indices are given by $k_0=1$, $s_0=3$, and $t_0=2$. In particular, the initial complement $F_1$ is given by
\begin{equation}
F_1=\mathcal{C}^{\circ}(\mathbf{e}_{3},\mathbf{e}_1-\mathbf{e}_2).
\end{equation}
Following the rules in \Cref{thm: recursive expression of complements}, we can obtain all the complement in $\mathfrak{D}_1$ by applying the linear maps $\Glinearmap_{1,1}$, $\Glinearmap_{[1]T}^{S}$, and $\Glinearmap_{[1]T}^{T}$. Their representation matrices with respect to the standard basis $[\mathbf{e}_1,\mathbf{e}_2,\mathbf{e}_3]$ are given by
\begin{equation}
\Glinearmap_{1,1}=\left(\begin{smallmatrix}
-2 & 0 & -1\\
0 & 1 & 0 \\
1 & 0 & 0
\end{smallmatrix}\right),
\quad
\Glinearmap_{[1]T}^{S}=\left(\begin{smallmatrix}
1 & 0 & 0\\
-2 & 0 & -1\\
2 & 1 & 2
\end{smallmatrix}\right),
\quad
\Glinearmap_{[1]T}^{T}=\left(\begin{smallmatrix}
0 & -1 & 0\\
-2 & 0 & -1\\
3 & 2 & 2
\end{smallmatrix}\right).
\end{equation}
\begin{figure}[htbp]
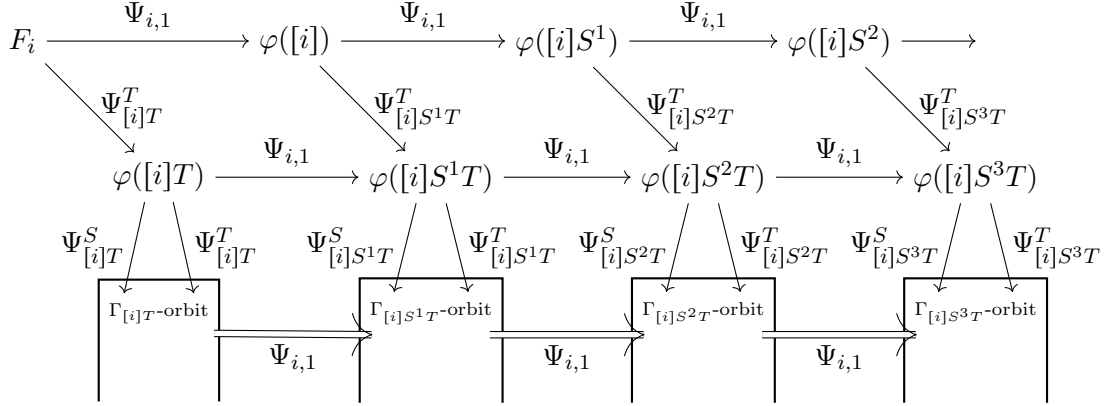

\centering
\subfile{complements_orbit}
\caption{Recursive process to obtain all the complements.}
\label{fig: recursive expression of complements}
\end{figure}
\end{example}

\section{Open problem}\label{sec: open problem}
In \cite{AC26}, we can observe that the $G$-fan structure of $B$-invariant type is a degeneration of the ones of {\em cluster-cyclic} exchange matrices. For example, we can find the following correspondence:
\begin{table}[htbp]
    \centering
    \begin{tabular}{c|cc}
        & $B$-invariant type (in this paper) & cluster-cyclic type \cite{AC26}
        \\
        \hline
        upper bound & half space (\Cref{lem: plane for the modified g-vectors}) & three hyperboloids \cite[Thm.~4.12]{AC26}
        \\
        complement & $2$-dimensional cones (\Cref{thm: pointwise expression of complements}) & $3$-dimensional cones \cite[Thm.~7.4]{AC26}
    \end{tabular}
\end{table}
\par
On the other hand, we have not formulated the counterpart of \Cref{thm: inner isomorphisms between subfans in the Markov case}, though we have already shown that they have the same tropical signs \cite{AC25b}. By considering the fact that all modified $g$-vectors are on one hyperboloid after an initial mutation, the following question naturally happens.
\begin{question}\label{Q: Mobious construction}
Can we get an analogous structure of \Cref{thm: inner isomorphisms between subfans in the Markov case} for the $G$-fan associated with the cluster-cyclic exchange matrices of rank $3$? Can we get the fractal structure via proper Lorentz transformations?
\end{question}
For the acyclic affine type, the corresponding $G$-fan is constructed by the {\em affine Weyl groups}, which has a geometric realization via the affine transformations. If \Cref{Q: Mobious construction} is answered, we might construct the analogy of this construction, and it will be helpful to understand explicit behavior of $G$-fans.

\subsection*{Acknowledgements} 
 The authors would like to sincerely thank Tomoki Nakanishi for his thoughtful guidance. The authors are also grateful to Yasuaki Gyoda, Salvatore Stella and Zhe Sun for their valuable discussions and insightful suggestions. In addition, Z. Chen  wants to thank Peigen Cao, Xiaowu Chen and Yu Ye for their help and support. R. Akagi is supported by JSPS KAKENHI Grant Number JP25KJ1438 and Chubei Itoh Foundation. Z. Chen is supported by National Natural Science Foundation of China (Grant No. 124B2003) and China Scholarship Council (Grant No. 202406340022).

\bibliography{Markov_G_fan}
\bibliographystyle{alpha}
\end{document}

%% file: Example.tex
\newcommand{\unitdist}{0.9}
\begin{tikzpicture}
\node (O) at (0,0) {$\underset{\tiny \textup{initial}}{(+,+,+)}$};
\node (i) at (0,-1.5) {$(\underset{k}{-},\underset{t}{+},\underset{s}{+})$};
\node (iS) at (0,-3) {$(\underset{s}{+},\underset{t}{+},\underset{k}{-})$};
\node (iSS) at (0,-4.5) {$(\underset{k}{-},\underset{t}{+},\underset{s}{+})$};
\node (iSSS) at (0,-6) {$(\underset{s}{+},\underset{t}{+},\underset{k}{-})$};

\node (iSST) at (2*\unitdist,-6) {$(\underset{t}{-},\underset{k}{-},\underset{s}{+})$};
\node (iSTS) at (4*\unitdist,-6) {$(\underset{k}{-},\underset{s}{+},\underset{t}{-})$};
\node (iSTT) at (6*\unitdist,-6) {$(\underset{t}{+},\underset{s}{-},\underset{k}{+})$};
\node (iTSS) at (8*\unitdist,-6) {$(\underset{t}{-},\underset{k}{-},\underset{s}{+})$};
\node (iTST) at (10*\unitdist,-6) {$(\underset{k}{+},\underset{t}{+},\underset{s}{-})$};
\node (iTTS) at (12*\unitdist,-6) {$(\underset{s}{-},\underset{k}{+},\underset{t}{+})$};
\node (iTTT) at (14*\unitdist,-6) {$(\underset{s}{+},\underset{t}{-},\underset{k}{-})$};

\node (iST) at (5*\unitdist,-4.5) {$(\underset{s}{+},\underset{k}{-},\underset{t}{-})$};
\node (iTS) at (9*\unitdist,-4.5) {$(\underset{t}{-},\underset{s}{+},\underset{k}{-})$};
\node (iTT) at (13*\unitdist,-4.5) {$(\underset{k}{+},\underset{s}{-},\underset{t}{+})$};
\node (iT) at (11*\unitdist,-3) {$(\underset{t}{-},\underset{k}{-},\underset{s}{+})$};

\draw[very thick] (iSSS.south west)--($(i.north west)+(0,0.1)$)--($(i.north east)+(0.05,0.1)$) node [right] {\small Trunk $\mathcal{T}^{<[i]S^{\infty}}$}--($(iSSS.south east)+(0.05,0)$);

\draw[preaction={draw=white, line width=4pt}] (-0.375,-0.5)--(-0.375,-1.2) node [pos=0.3, auto=left, inner sep=1pt] {$i=1$};
\draw[->] (-0.375,-0.5)->(-0.375,-1.2);
\draw[->] (i.-130)->(iS.130) node [pos=0.5, auto=left, inner sep=1pt] {$S=3$};
\draw[->] (iS.-130)->(iSS.130) node [pos=0.5, auto=left, inner sep=1pt] {$S=1$};
\draw[->] (iSS.-130)->(iSSS.130) node [pos=0.5, auto=left, inner sep=1pt] {$S=3$};

\draw[preaction={draw=white, line width=4pt} ,->] (i.-10)->(iT.170) node [pos=0.5, auto=left, inner sep=1pt] {$T=2$}; 
\draw[preaction={draw=white, line width=4pt} ,->] (iS.-10)->(iST.170) node [pos=0.5, auto=left, inner sep=1pt] {$T=2$}; 
\draw[preaction={draw=white, line width=4pt} ,->] (iSS.-10)->(iSST.100) node [pos=0.5, auto=left, inner sep=1pt] {$T=2$};

\draw[->] (iT.-150)->(iTS.45) node [pos=0.5, auto=right, inner sep=1pt] {$S=3$};
\draw[->] (iT.-30)->(iTT.135) node [pos=0.5, auto=left, inner sep=1pt] {$T=1$};
\draw[->] (iST.-130)->(iSTS.90) node [pos=0.5, auto=right, inner sep=1pt] {$S=1$};
\draw[->] (iST.-50)->(iSTT.90) node [pos=0.5, auto=left, inner sep=1pt] {$T=3$};
\draw[->] (iTS.-130)->(iTSS.90) node [pos=0.5, auto=right, inner sep=1pt] {$S=2$};
\draw[->] (iTS.-50)->(iTST.90) node [pos=0.5, auto=left, inner sep=1pt] {$T=1$};
\draw[->] (iTT.-130)->(iTTS.90) node [pos=0.5, auto=right, inner sep=1pt] {$S=2$};
\draw[->] (iTT.-50)->(iTTT.90) node [pos=0.5, auto=left, inner sep=1pt] {$T=3$};
\end{tikzpicture}

%% file: alphabeta_after_11.tex
\newcommand{\unitdist}{1}
\newcommand{\unithight}{1}
\begin{tikzpicture}
\node (SSSX) at (-7*\unitdist,0) {$(4,1)$};
\node (TSSX) at (-5*\unitdist,0) {$(1,4)$};
\node (STSX) at (-3*\unitdist,0) {$(4,3)$};
\node (TTSX) at (-1*\unitdist,0) {$(3,4)$};
\node (SSTX) at (1*\unitdist,0) {$(5,2)$};
\node (TSTX) at (3*\unitdist,0) {$(2,5)$};
\node (STTX) at (5*\unitdist,0) {$(5,3)$};
\node (TTTX) at (7*\unitdist,0) {$(3,5)$};

\node (SSX) at (-6*\unitdist,1*\unithight) {$(3,1)$};
\node (TSX) at (-2*\unitdist,1*\unithight) {$(1,3)$};
\node (STX) at (2*\unitdist,1*\unithight) {$(3,2)$};
\node (TTX) at (6*\unitdist,1*\unithight) {$(2,3)$};

\node (SX) at (-4*\unitdist,2*\unithight) {$(2,1)$};
\node (TX) at (4*\unitdist,2*\unithight) {$(1,2)$};
\node (X) at (0,3*\unithight) {$(1,1)$};

\node at (X) [yshift=-9] {\tiny $1_{\mathcal{M}}$}; 
\foreach \x in {T,S,TT,TS,ST,SS,TTT,TTS,TST,TSS,STT,STS,SST,SSS}
{
\node at (\x X) [yshift=-9] {\tiny $\x$};
}

\draw[->] (X)->(SX) node [pos=0.5, above left] {$S\times$};
\draw[->] (X)->(TX) node [pos=0.5, above right] {$T \times$};
\draw[->] (SX)->(SSX) node [pos=0.5, above left] {$S\times$};
\draw[->] (SX)->(TSX) node [pos=0.5, above right] {$T\times$};
\draw[->] (TX)->(STX) node [pos=0.5, above left] {$S\times$};
\draw[->] (TX)->(TTX) node [pos=0.5, above right] {$T\times$};

\draw[->] (SSX)->(SSSX) node [pos=0.5, above left] {$S\times$};
\draw[->] (SSX)->(TSSX) node [pos=0.5, above right] {$T \times$};
\draw[->] (TSX)->(STSX) node [pos=0.5, above left] {$S\times$};
\draw[->] (TSX)->(TTSX) node [pos=0.5, above right] {$T\times$};
\draw[->] (STX)->(SSTX) node [pos=0.5, above left] {$S\times$};
\draw[->] (STX)->(TSTX) node [pos=0.5, above right] {$T\times$};
\draw[->] (TTX)->(STTX) node [pos=0.5, above left] {$S\times$};
\draw[->] (TTX)->(TTTX) node [pos=0.5, above right] {$T\times$};
\end{tikzpicture}

%% file: S_mutations.tex
\begin{tikzpicture}
\coordinate (T) at (0,0);
\fill (T) circle [radius=0.06];
\foreach \x in {0,1,2,3,4,5,6,7}
{
\coordinate (A\x) at (2*\x-1.5,2);
\fill (A\x) circle [radius=0.06];
\draw (T)--(A\x);
}
\draw[thick, red] (A0)--(A7);
\draw[dashed, blue] (T)--(12.5,0);
\foreach \x in {8,...,20}
{
\draw (T)--(12.5,{25/(2*\x-1.5)});
}
\node at (T) [xshift=-2, yshift=12] {$T$};
\node at (A0) [xshift=15, yshift=-8] {$S$};
\node at (A1) [xshift=-12, yshift=-8] {$K$};
\node at (A1) [xshift=7, yshift=-8] {$S$};
\node at (A2) [xshift=-20, yshift=-8] {$K$};
\node at (A2) [xshift=0, yshift=-8] {$S$};
\node at (A3) [xshift=-30, yshift=-8] {$K$};
\node at (T) [below] {$\tilde{\mathbf{g}}_{T}^{\mathbf{w}}=\tilde{\mathbf{g}}_{T}^{\mathbf{w}S^n}$};

\node at (-0.25,1) {$\mathbf{w}$};
\node at (0.7,1) {$\mathbf{w}S$};
\node at (1.7,1) {$\mathbf{w}S^2$};
\end{tikzpicture}

%% file: proof_of_fractal_theorem.tex
\begin{tikzpicture}
\node[fill, circle, minimum size=0.12cm, inner sep=0pt] (O) at (0,0) {};
\node[fill, circle, minimum size=0.12cm, inner sep=0pt] (Sofw) at (2,0) {};
\node[fill, circle, minimum size=0.12cm, inner sep=0pt] (Tofw) at (1.5,1) {};
\node[fill, circle, minimum size=0.12cm, inner sep=0pt] (Kofw) at (3.5,1) {};

\node at (O) [below left] {$\mathfrak{g}_i$};

\draw[->] (O)->(Tofw) node [pos=0.5, above left] {$\tilde{\mathbf{v}}_{SK}^{\mathbf{w}}$};
\draw[->] (O)->(Sofw) node [pos=0.5, below] {$\tilde{\mathbf{v}}_{TK}^{\mathbf{w}}$};
\draw[->] (Tofw)->(Kofw);
\draw[->] (Sofw)->(Kofw);

\draw (Sofw)--(Tofw);

\node at (2.33,0.66) {\tiny $\mathbf{w}$};
\node at (Kofw) [xshift=-15, yshift=-5] {\tiny $K$};
\node at (Sofw) [xshift=2, yshift=7] {\tiny $S$};
\node at (Tofw) [xshift=8, yshift=-5] {\tiny $T$};

\coordinate (SofwS) at (Kofw);
\coordinate (TofwS) at (Tofw);
\node[fill, circle, minimum size=0.12cm, inner sep=0pt] (KofwS) at (5,2) {};

\draw[->] (SofwS)->(KofwS);
\draw[->] (TofwS)->(KofwS);

\node at (3.333,1.4) {\tiny $\mathbf{w}S$};
\node at (KofwS) [xshift=-22, yshift=-10] {\tiny $K$};
\node at (SofwS) [xshift=0, yshift=5] {\tiny $S$};
\node at (TofwS) [xshift=30, yshift=5] {\tiny $T$};

\coordinate (SofwT) at (Kofw);
\coordinate (TofwT) at (Sofw);
\node[fill, circle, minimum size=0.12cm, inner sep=0pt] (KofwT) at (5.5,1) {};

\draw[->] (SofwT)->(KofwT);
\draw[->] (TofwT)->(KofwT);

\node at (3.666,0.6) {\tiny $\mathbf{w}T$};
\node at (KofwT) [xshift=-30, yshift=-5] {\tiny $K$};
\node at (SofwT) [yshift=-5] {\tiny $S$};
\node at (TofwT) [xshift=22, yshift=10] {\tiny $T$};

\draw[red] (3.5,1)--(7,2) node [above right] {$\mathcal{C}^{\circ}(\tilde{\mathbf{g}}_{K}^{\mathbf{w}},\tilde{\mathbf{v}}_{SK}^{\mathbf{w}}+\tilde{\mathbf{v}}_{TK}^{\mathbf{w}})$};
\draw[blue, dashed] (1.5,1)--(4.5,3);
\draw[blue, dashed] (2,0)--(7,0);

\fill[gray, opacity=0.2] (1.5,1)--(4.5,3)--(7,3)--(7,2)--(3.5,1)--cycle;
\fill[gray, opacity=0.2] (3.5,1)--(7,2)--(7,0)--(2,0)--cycle;

\node at (KofwS) [above] {$\mathcal{U}_{\circ}^{\mathbf{w}S}$};
\node at (KofwT) [right] {$\mathcal{U}_{\circ}^{\mathbf{w}T}$};
\end{tikzpicture}

%% file: proof_of_pointwise_complement_expressions.tex
\newcommand{\unitlengthproof}{1.5}
\begin{tikzpicture}
\clip (-5,-3) rectangle (10,0.5);
\node[fill, circle, minimum size=0.12cm, inner sep=0pt] (O) at (0,0) {};
\node at (O) [xshift=-10] {$\mathfrak{g}_{i}$};

\foreach \x in {0,1,2,3,4,5,6,7,8,9,10}
{
\node[fill, circle, minimum size=0.12cm, inner sep=0pt] (A\x) at (-\unitlengthproof+\unitlengthproof*\x,-1) {};
\draw[->] (O)->(A\x);
\coordinate (B\x) at (-4*\unitlengthproof+4*\unitlengthproof*\x,-4);
\draw[dashed] (A\x)--(B\x);
}

\foreach \x [evaluate=\x as \nextx using int(\x+1)] in {3,4,5,6,7,8,9}
{
\draw[red, thick, ->] (A\x)->(A\nextx);
}

\draw[red, thick, ->] (A0)->(A1) node [pos=0.7, above] {\tiny $\tilde{\mathbf{v}}_{k_0s_0}$};
\draw[red, thick, ->] (A1)->(A2) node [pos=0.4, above] {\tiny $\tilde{\mathbf{v}}_{k_0s_0}$};
\draw[red, thick, ->] (A2)->(A3) node [pos=0.1, above] {\tiny $\tilde{\mathbf{v}}_{k_0s_0}$};

\foreach \x [evaluate=\x as \nextx using int(\x+1)] in {0,1,2,3,4,5,6,7,8,9}
{
\node[fill, circle, minimum size=0.12cm, inner sep=0pt] (C\x) at (-\unitlengthproof+2*\unitlengthproof*\x,-2) {};
\draw (A\x)--(C\x)--(A\nextx);
}

\node at (C0) [below] {$\mathcal{U}^{[i]T}$};
\node at (C1) [below] {$\mathcal{U}^{[i]ST}$};
\node at (C2) [below right] {$\mathcal{U}^{[i]S^2T}$};
\node at (C3) [below right] {$\mathcal{U}^{[i]S^3T}$};


\node at (A0) [above left] {$\tilde{\mathbf{e}}_{s_0}$};
\draw[red, thick, dashed] (A0)--(B0);
\draw[red, thick, ->, dashed] (A0)->(-1.7*\unitlengthproof,-1.7) node [below right] {\tiny $\tilde{\mathbf{v}}_{t_0k_0}$}; 
\end{tikzpicture}

%% file: complements.tex
\begin{tikzpicture}[scale=0.4]
\clip (-4,-14.2) rectangle (24.2,14.2);

\fill (0,0) circle [radius=0.12] node [left] {$\mathfrak{g}_i$};
\draw[->] (0,0)->(1,2) node [left] {$\tilde{\mathbf{v}}_{t_0k_0}$};
\draw[->] (0,0)->(1,-2) node [left] {$\tilde{\mathbf{v}}_{k_0s_0}$};
\fill (1,-2) circle [radius=0.12];
\fill (2,-4) circle [radius=0.12];
\fill (3,-6) circle [radius=0.12];
\fill (4,-8) circle [radius=0.12];
\fill (5,-10) circle [radius=0.12];
\fill[red] (1,2) circle [radius=0.12];
\draw[dashed, red, ultra thick] (1,2)--(5.10000000000000,10.2000000000000) node [above] {$\frac{0}{1}$};
\fill[red] (2,0) circle [radius=0.12];
\draw[dashed, red, ultra thick] (2,0)--(20.2000000000000,0.000000000000000) node [right] {$\frac{1}{1}$};
\fill[red] (3,-2) circle [radius=0.12];
\draw[dashed, red, ultra thick] (3,-2)--(15.3000000000000,-10.2000000000000) node [below] {$\frac{2}{1}$};
\fill[red] (4,-4) circle [radius=0.12];
\draw[dashed, red, ultra thick] (4,-4)--(10.2000000000000,-10.2000000000000) node [below] {$\frac{3}{1}$};
\fill[red] (5,-6) circle [radius=0.12];
\draw[dashed, red, ultra thick] (5,-6)--(8.50000000000000,-10.2000000000000) node [below] {$\frac{4}{1}$};
\fill[red] (6,-8) circle [radius=0.12];
\draw[dashed, red, ultra thick] (6,-8)--(7.65000000000000,-10.2000000000000) node [below] {$\frac{5}{1}$};
\fill[red] (7,-10) circle [radius=0.12];
\draw[dashed, red, ultra thick] (7,-10)--(7.14000000000000,-10.2000000000000);
\fill (2,4) circle [radius=0.12];
\fill[red] (3,2) circle [radius=0.12];
\draw[dashed, red] (3,2)--(15.3000000000000,10.2000000000000) node [above] {$\frac{1}{2}$};
\fill (4,0) circle [radius=0.12];
\fill[red] (5,-2) circle [radius=0.12];
\draw[dashed, red] (5,-2)--(20.2000000000000,-8.08000000000000) node [right] {$\frac{3}{2}$};
\fill (6,-4) circle [radius=0.12];
\fill[red] (7,-6) circle [radius=0.12];
\draw[dashed, red] (7,-6)--(11.9000000000000,-10.2000000000000) node [below] {$\frac{5}{2}$};
\fill (8,-8) circle [radius=0.12];
\fill[red] (9,-10) circle [radius=0.12];
\draw[dashed, red] (9,-10)--(9.18000000000000,-10.2000000000000);
\fill (3,6) circle [radius=0.12];
\fill[red] (4,4) circle [radius=0.12];
\draw[dashed, red] (4,4)--(10.2000000000000,10.2000000000000) node [above] {$\frac{1}{3}$};
\fill[red] (5,2) circle [radius=0.12];
\draw[dashed, red] (5,2)--(20.2000000000000,8.08000000000000) node [right] {$\frac{2}{3}$};
\fill (6,0) circle [radius=0.12];
\fill[red] (7,-2) circle [radius=0.12];
\draw[dashed, red] (7,-2)--(20.2000000000000,-5.77142857142857) node [right] {$\frac{4}{3}$};
\fill[red] (8,-4) circle [radius=0.12];
\draw[dashed, red] (8,-4)--(20.2000000000000,-10.1000000000000) node [right] {$\frac{5}{3}$};
\fill (9,-6) circle [radius=0.12];
\fill[red] (10,-8) circle [radius=0.12];
\draw[dashed, red] (10,-8)--(12.7500000000000,-10.2000000000000);
\fill[red] (11,-10) circle [radius=0.12];
\draw[dashed, red] (11,-10)--(11.2200000000000,-10.2000000000000);
\fill (4,8) circle [radius=0.12];
\fill[red] (5,6) circle [radius=0.12];
\draw[dashed, red] (5,6)--(8.50000000000000,10.2000000000000) node [above] {$\frac{1}{4}$};
\fill (6,4) circle [radius=0.12];
\fill[red] (7,2) circle [radius=0.12];
\draw[dashed, red] (7,2)--(20.2000000000000,5.77142857142857) node [right] {$\frac{3}{4}$};
\fill (8,0) circle [radius=0.12];
\fill[red] (9,-2) circle [radius=0.12];
\draw[dashed, red] (9,-2)--(20.2000000000000,-4.48888888888889) node [right] {$\frac{5}{4}$};
\fill (10,-4) circle [radius=0.12];
\fill[red] (11,-6) circle [radius=0.12];
\draw[dashed, red] (11,-6)--(18.7000000000000,-10.2000000000000);
\fill (12,-8) circle [radius=0.12];
\fill[red] (13,-10) circle [radius=0.12];
\draw[dashed, red] (13,-10)--(13.2600000000000,-10.2000000000000);
\fill (5,10) circle [radius=0.12];
\fill[red] (6,8) circle [radius=0.12];
\draw[dashed, red] (6,8)--(7.65000000000000,10.2000000000000) node [above] {$\frac{1}{5}$};
\fill[red] (7,6) circle [radius=0.12];
\draw[dashed, red] (7,6)--(11.9000000000000,10.2000000000000) node [above] {$\frac{2}{5}$};
\fill[red] (8,4) circle [radius=0.12];
\draw[dashed, red] (8,4)--(20.2000000000000,10.1000000000000) node [right] {$\frac{3}{5}$};
\fill[red] (9,2) circle [radius=0.12];
\draw[dashed, red] (9,2)--(20.2000000000000,4.48888888888889) node [right] {$\frac{4}{5}$};
\fill (10,0) circle [radius=0.12];
\fill[red] (11,-2) circle [radius=0.12];
\draw[dashed, red] (11,-2)--(20.2000000000000,-3.67272727272727);
\fill[red] (12,-4) circle [radius=0.12];
\draw[dashed, red] (12,-4)--(20.2000000000000,-6.73333333333333);
\fill[red] (13,-6) circle [radius=0.12];
\draw[dashed, red] (13,-6)--(20.2000000000000,-9.32307692307692);
\fill[red] (14,-8) circle [radius=0.12];
\draw[dashed, red] (14,-8)--(17.8500000000000,-10.2000000000000);
\fill (15,-10) circle [radius=0.12];
\fill[red] (7,10) circle [radius=0.12];
\draw[dashed, red] (7,10)--(7.14000000000000,10.2000000000000);
\fill (8,8) circle [radius=0.12];
\fill (9,6) circle [radius=0.12];
\fill (10,4) circle [radius=0.12];
\fill[red] (11,2) circle [radius=0.12];
\draw[dashed, red] (11,2)--(20.2000000000000,3.67272727272727);
\fill (12,0) circle [radius=0.12];
\fill[red] (13,-2) circle [radius=0.12];
\draw[dashed, red] (13,-2)--(20.2000000000000,-3.10769230769231);
\fill (14,-4) circle [radius=0.12];
\fill (15,-6) circle [radius=0.12];
\fill (16,-8) circle [radius=0.12];
\fill[red] (17,-10) circle [radius=0.12];
\draw[dashed, red] (17,-10)--(17.3400000000000,-10.2000000000000);
\fill[red] (9,10) circle [radius=0.12];
\draw[dashed, red] (9,10)--(9.18000000000000,10.2000000000000);
\fill[red] (10,8) circle [radius=0.12];
\draw[dashed, red] (10,8)--(12.7500000000000,10.2000000000000);
\fill[red] (11,6) circle [radius=0.12];
\draw[dashed, red] (11,6)--(18.7000000000000,10.2000000000000);
\fill[red] (12,4) circle [radius=0.12];
\draw[dashed, red] (12,4)--(20.2000000000000,6.73333333333333);
\fill[red] (13,2) circle [radius=0.12];
\draw[dashed, red] (13,2)--(20.2000000000000,3.10769230769231);
\fill (14,0) circle [radius=0.12];
\fill[red] (15,-2) circle [radius=0.12];
\draw[dashed, red] (15,-2)--(20.2000000000000,-2.69333333333333);
\fill[red] (16,-4) circle [radius=0.12];
\draw[dashed, red] (16,-4)--(20.2000000000000,-5.05000000000000);
\fill[red] (17,-6) circle [radius=0.12];
\draw[dashed, red] (17,-6)--(20.2000000000000,-7.12941176470588);
\fill[red] (18,-8) circle [radius=0.12];
\draw[dashed, red] (18,-8)--(20.2000000000000,-8.97777777777778);
\fill[red] (19,-10) circle [radius=0.12];
\draw[dashed, red] (19,-10)--(19.3800000000000,-10.2000000000000);
\fill[red] (11,10) circle [radius=0.12];
\draw[dashed, red] (11,10)--(11.2200000000000,10.2000000000000);
\fill (12,8) circle [radius=0.12];
\fill[red] (13,6) circle [radius=0.12];
\draw[dashed, red] (13,6)--(20.2000000000000,9.32307692307692);
\fill (14,4) circle [radius=0.12];
\fill[red] (15,2) circle [radius=0.12];
\draw[dashed, red] (15,2)--(20.2000000000000,2.69333333333333);
\fill (16,0) circle [radius=0.12];
\fill[red] (17,-2) circle [radius=0.12];
\draw[dashed, red] (17,-2)--(20.2000000000000,-2.37647058823529);
\fill (18,-4) circle [radius=0.12];
\fill[red] (19,-6) circle [radius=0.12];
\draw[dashed, red] (19,-6)--(20.2000000000000,-6.37894736842105);
\fill (20,-8) circle [radius=0.12];
\fill[red] (13,10) circle [radius=0.12];
\draw[dashed, red] (13,10)--(13.2600000000000,10.2000000000000);
\fill[red] (14,8) circle [radius=0.12];
\draw[dashed, red] (14,8)--(17.8500000000000,10.2000000000000);
\fill (15,6) circle [radius=0.12];
\fill[red] (16,4) circle [radius=0.12];
\draw[dashed, red] (16,4)--(20.2000000000000,5.05000000000000);
\fill[red] (17,2) circle [radius=0.12];
\draw[dashed, red] (17,2)--(20.2000000000000,2.37647058823529);
\fill (18,0) circle [radius=0.12];
\fill[red] (19,-2) circle [radius=0.12];
\draw[dashed, red] (19,-2)--(20.2000000000000,-2.12631578947368);
\fill[red] (20,-4) circle [radius=0.12];
\draw[dashed, red] (20,-4)--(20.2000000000000,-4.04000000000000);
\fill (15,10) circle [radius=0.12];
\fill (16,8) circle [radius=0.12];
\fill[red] (17,6) circle [radius=0.12];
\draw[dashed, red] (17,6)--(20.2000000000000,7.12941176470588);
\fill (18,4) circle [radius=0.12];
\fill[red] (19,2) circle [radius=0.12];
\draw[dashed, red] (19,2)--(20.2000000000000,2.12631578947368);
\fill (20,0) circle [radius=0.12];
\fill[red] (17,10) circle [radius=0.12];
\draw[dashed, red] (17,10)--(17.3400000000000,10.2000000000000);
\fill[red] (18,8) circle [radius=0.12];
\draw[dashed, red] (18,8)--(20.2000000000000,8.97777777777778);
\fill[red] (19,6) circle [radius=0.12];
\draw[dashed, red] (19,6)--(20.2000000000000,6.37894736842105);
\fill[red] (20,4) circle [radius=0.12];
\draw[dashed, red] (20,4)--(20.2000000000000,4.04000000000000);
\fill[red] (19,10) circle [radius=0.12];
\draw[dashed, red] (19,10)--(19.3800000000000,10.2000000000000);
\fill (20,8) circle [radius=0.12];
\end{tikzpicture}

%% file: complements_orbit.tex
\begin{tikzpicture}[scale=0.9]
\node (O) at (0,0) {$F_i$};
\node (A0) at (4,0) {$\varphi([i])$};
\foreach \x in {1,2}
{
\node (A\x) at (4*\x+4,0) {$\varphi([i]S^{\x})$};
}

\node (B0) at (2,-2) {$\varphi([i]T)$};
\node (C0) at (2,-4) {\tiny $\Gamma_{[i]T}$-orbit};
\draw[thick] ($(C0.south west)+(0,-1)$)--($(C0.north west)+(0,0.2)$)--($(C0.north east)+(0,0.2)$)--($(C0.south east)+(0,-1)$);

\foreach \x in {1,2,3}
{
\node (B\x) at (4*\x+2,-2) {$\varphi([i]S^{\x}T)$};
\node (C\x) at (4*\x+2,-4) {\tiny $\Gamma_{[i]S^{\x}T}$-orbit};
\draw[thick] ($(C\x.south west)+(0,-1)$)--($(C\x.north west)+(0,0.2)$)--($(C\x.north east)+(0,0.2)$)--($(C\x.south east)+(0,-1)$);
}

\draw[->] (B0.-120)->(C0.150) node [pos=0.5, left] {$\Glinearmap_{[i]T}^{S}$};
\draw[->] (B0.-60)->(C0.30) node [pos=0.5, right] {$\Glinearmap_{[i]T}^{T}$};
\foreach \x in {1,2,3}
{
\draw[->] (B\x.-120)->(C\x.150) node [pos=0.5, left] {$\Glinearmap_{[i]S^{\x}T}^{S}$};
\draw[->] (B\x.-60)->(C\x.30) node [pos=0.5, right] {$\Glinearmap_{[i]S^{\x}T}^{T}$};
}
\draw[->] (O)->(A0) node [pos=0.5, above] {$\Psi_{i,1}$};
\draw[->] (O)->(B0) node [pos=0.5, right] {$\Psi_{[i]T}^{T}$};
\foreach \x [evaluate=\x as \nextx using int(\x+1)] in {0,1}
{
\draw[->] (A\x)->(A\nextx) node [pos=0.5, above] {$\Psi_{i,1}$};
}
\foreach \x [evaluate=\x as \nextx using int(\x+1)] in {0,1,2}
{
\draw[->] (A\x)->(B\nextx) node [pos=0.5, right] {$\Glinearmap_{[i]S^{\nextx}T}^{T}$};
\draw[->] (B\x)->(B\nextx) node [pos=0.5, above] {$\Glinearmap_{i,1}$};
\draw[->, double distance=2] (C\x.-20)->(C\nextx.-160) node [pos=0.5, below] {$\Glinearmap_{i,1}$};
}
\draw[->] (A2)->(14,0);
\end{tikzpicture}